\documentclass[oneside, reqno, 11pt, a4paper]{amsart}

\usepackage[english]{babel}

\AtBeginDocument{
  \DeclareSymbolFont{AMSb}{U}{msb}{m}{n}
  \DeclareSymbolFontAlphabet{\mathbb}{AMSb}
}
\DeclareFontFamily{U}{mathx}{\hyphenchar\font45}
\DeclareFontShape{U}{mathx}{m}{n}{<-> mathx10}{}
\DeclareSymbolFont{mathx}{U}{mathx}{m}{n}
\DeclareMathAccent{\widebar}{0}{mathx}{"73}

\usepackage{graphicx}
\usepackage{svg}
\usepackage{tikz}
\usepackage{caption}
\usepackage{aligned-overset}
\usepackage{subcaption}
\tikzstyle{arrow} = [thick,->,>=stealth]
\usetikzlibrary{shapes.misc,matrix,fit,positioning,arrows.meta,decorations.pathreplacing,calc,shapes.geometric,arrows,shapes,patterns,math}
\tikzset{Matrix/.style={matrix of nodes, font=\footnotesize,text height=1pt, text depth=0.5pt, text width=8.5pt, align=center, column sep=0pt, row sep=0pt, nodes in empty cells}}

\graphicspath{{Figures/}{Figures/results_3d/}{Figures/results/}{../plots/}}

\usepackage{acronym}
\acrodef{fem}[FEM]{Finite Element Method}
\acrodef{fems}[FEMs]{Finite Element Methods}
\acrodef{gmg}[GMG]{Geometric Multigrid Method}
\acrodef{gmgs}[GMGs]{Geometric Multigrid Methods}
\acrodef{ddm}[DDM]{Domain Decomposition Method}
\acrodef{ddms}[DDMs]{Domain Decomposition Methods}
\acrodef{bddc}[BDDC]{Balancing Domain Decomposition by Constraints}
\acrodef{feti}[FETI]{Finite Element Tearing and Interconnecting}
\acrodef{feti-dp}[FETI-DP]{Dual-Primal Finite Element Tearing and Interconnecting}
\acrodef{hho}[HHO]{Hybrid High-Order}
\acrodef{hdg}[HDG]{Hybridizable Discontinuous Galerkin}
\acrodef{vem}[VEM]{Virtual Element Method}
\acrodef{wg}[WG]{Weak Galerkin}
\acrodef{dg}[DG]{Discontinuous Galerkin}
\acrodef{bnn}[BNN]{Balancing Neumann-Neumann}
\acrodef{pde}[PDE]{Partial Differential Equation}
\acrodef{pdes}[PDEs]{Partial Differential Equations}
\acrodef{dof}[DOF]{Degree of Freedom}
\acrodef{dofs}[DOFs]{Degrees of Freedom}

\usepackage[a4paper, pdftex, left=2cm, top=2cm, right=2cm, bottom=2cm]{geometry}
\usepackage[final]{pdfpages}
\usepackage{changepage}

\usepackage[foot]{amsaddr}
\numberwithin{equation}{section}

\usepackage{url}
\usepackage{hyperref}
\hypersetup{
    breaklinks=true,
    bookmarksopen=true,
    pdftitle={BDDC solvers for HHO}, 
    pdfauthor={Santiago Badia, Jerome Droniou, Jai Tushar}, 
	pdfauthor={}
    pdfsubject={}, 
    colorlinks=true,
    linkcolor=black,
    citecolor=blue,
    filecolor=black,
    urlcolor=blue
}

\usepackage[normalem]{ulem}
\usepackage{csquotes}
\usepackage[backend=biber, defernumbers=true, maxbibnames=5, style=numeric-comp, isbn=false, bibencoding=utf8, safeinputenc, url=false, doi=true, giveninits=true]{biblatex}
\usepackage{mathtools}
\usepackage{mathabx}

\usepackage{float}
\usepackage{framed}
\usepackage{verbatim}
\usepackage{fancyvrb}
\usepackage{booktabs}
\usepackage{colortbl}
\usepackage{multirow}
\usepackage{algorithm}
\usepackage{algpseudocode}
\usepackage{cancel}
\usepackage{accents}
\usepackage{mleftright}
\usepackage{thm-restate}

\makeatletter
\algdef{S}[IF]{IfNoThen}[1]{\algorithmicif\ #1}
\makeatother
\usepackage[inline]{enumitem}
\usepackage{siunitx}

\newtheorem{theorem}{Theorem}
\numberwithin{theorem}{section}
\newtheorem{remark}[theorem]{Remark}
\newtheorem{lemma}[theorem]{Lemma}
\newtheorem{corollary}[theorem]{Corollary}
\newtheorem{proposition}[theorem]{Proposition}
\newtheorem{assumption}[theorem]{Assumption}
\usepackage[breakable]{tcolorbox}
\tcbuselibrary{theorems}
\tcbuselibrary{skins}
\tcbset{
	commonstyle/.style={
		theorem style=plain,
		enhanced jigsaw,
		fonttitle=\bfseries,
		fontupper=\itshape,
		halign=justify,
		separator sign=:,
		description delimiters none,
		description font=\bfseries, 
		terminator sign={.\hspace{0.25em}},
		arc=0mm,outer arc=0mm,
		boxrule=0pt,toprule=0pt,bottomrule=0pt,leftrule=0pt,rightrule=0pt,
		titlerule=0pt,toptitle=0pt,bottomtitle=0pt,top=0pt,
		colback=white,coltitle=black,
		boxsep=0pt, bottom=0pt, left=0pt, 
	}
}
\newtcbtheorem[]{myproblem}{Problem}%
{center, commonstyle, fonttitle=\bfseries}{pb}
\newtcbtheorem[]{mydefinition}{Definition}
{center, commonstyle, fonttitle=\bfseries}{pb}
\newtcbtheorem[]{myassumption}{Assumption}
{center, commonstyle, fonttitle=\bfseries}{pb}

\usepackage{amsmath, amsfonts, amssymb, amscd, bm, mathtools}
\newcommand{\Mh}{\mathcal{M}_{h}}
\newcommand{\Th}{\mathcal{T}_{h}}

\newcommand{\Fh}{\mathcal{F}_{h}}

\newcommand{\interior}{{\rm in}}
\newcommand{\boundary}{{\rm bd}}
\newcommand{\Fhi}{\Fh^{\interior}}
\newcommand{\Fhb}{\Fh^{\boundary}}

\newcommand{\FFhb}{\mathcal{F\!F}_h^{\boundary}}

\newcommand{\FFh}{\mathcal{F\!F}_h}
\newcommand{\Eh}{\mathcal{E}_{h}}

\newcommand{\MH}{\mathcal{M}_{H}}
\renewcommand{\TH}{\mathcal{T}_{H}}
\newcommand{\FH}{\mathcal{F}_{H}}

\newcommand{\wh}[1]{\widehat{#1}}
\newcommand{\wt}[1]{\widetilde{#1}} 
\newcommand{\ul}[1]{\underline{#1}}
\newcommand{\Uhbddc}{\ul{\wt{U}}_h}
\newcommand{\Uhbub}{\ul{U}_{h,0}}
\newcommand{\Uhbubb}{U_{h,0}^\partial}
\newcommand{\Uhc}{\wh{\ul{U}}_{h}}
\newcommand{\Uh}{\Uhc}

\newcommand{\Uhi}{U_{h}}

\newcommand{\Uhcb}{\wh{U}_{h}^{\partial}}
\newcommand{\Uhb}{\Uhcb}
\newcommand{\Uhcbb}{\wh{U}_{h}^{\partial,\boundary}}
\newcommand{\Uhbb}{\Uhcbb}

\newcommand{\bs}[1]{\boldsymbol{#1}}

\newcommand{\Uhd}{\ul{U}_{h}}

\newcommand{\Uhdb}{U_{h}^{\partial}}

\newcommand{\tr}{\gamma}

\newcommand{\dffp}{|x_f-x_{f'}|}
\newcommand{\dfg}{|x_f-x_{g}|}

\newcommand{\dfGbd}{\mbox{dist}(x_f,\partial\Gamma)}
\newcommand{\dist}{\mathop{\rm dist}}

\newcommand{\norm}[2]{\|#2\|_{#1}}
\newcommand{\seminorm}[2]{|#2|_{#1}}
\newcommand{\tnorm}[2]{\vert #2\vert_{#1}}

\newcommand{\dom}{\Omega}

\newcommand{\ol}[1]{\overline{#1}}

\newcommand{\A}{\mathcal{A}}

\newcommand{\G}{\mathcal{G}}

\newcommand{\Co}{\mathcal{C}}

\newcommand{\W}{\mathcal{W}}

\newcommand{\Ift}{\mathcal{I}_{ft}}

\newcommand{\HhalfOp}{\boldsymbol{\mathcal{H}}}
\newcommand{\U}{\mathcal{U}}
\newcommand{\V}{\mathcal{V}}
\newcommand{\Q}{\boldsymbol{\mathcal{Q}}}

\title[]{
Analysis of BDDC preconditioners for non-conforming polytopal hybrid discretisation methods
} 
\date{\today}

\address{$^\dagger$School of mathematics\\Monash university\\Clayton\\Victoria 3800\\Australia}
\address{$^\sharp$IMAG, UMR CNRS 5149, Montpellier, France.}
\author[S. Badia]{Santiago Badia$^{\dagger}$\textsuperscript{*}}
\thanks{\textsuperscript{*}Corresponding author.}
\email{santiago.badia@monash.edu}
\author[J. Droniou]{Jerome Droniou$^{\sharp\dagger}$}
\email{jerome.droniou@umontpellier.fr}
\author[J. Manyer]{Jordi Manyer$^{\dagger}$}
\email{jordi.manyer@monash.edu}
\author[J. Tushar]{Jai Tushar$^{\dagger}$}
\email{jai.tushar@monash.edu}

\begin{document}

\begin{abstract}
	In this work, we build on the discrete trace theory developed by Badia, Droniou, and Tushar (Foundations of Computational Mathematics, in press, 2025; \href{https://doi.org/10.1007/s10208-025-09734-6}{doi:10.1007/s10208-025-09734-6}) to analyze the convergence rate of the Balancing Domain Decomposition by Constraints (BDDC) preconditioner generated from non-conforming polytopal hybrid discretizations. We prove polylogarithmic bounds on the condition number for the preconditioner that are independent of the mesh parameter and the number of subdomains, and that hold on polytopal meshes. The analysis relies on the continuity of a face truncation operator, which we establish in the fully discrete polytopal setting. To validate the theory, we present numerical experiments that confirm the truncation estimate and condition number bounds. In particular, we conduct weak scalability tests for second-order elliptic problems discretized using discontinuous skeletal methods, specifically Hybridizable Discontinuous Galerkin (HDG) and Hybrid High-Order (HHO) methods. We also demonstrate the robustness of the preconditioner for piecewise discontinuous coefficients with large jumps. 
\end{abstract}


\maketitle

%
\section{Introduction}\label{sec:introduction}

Polytopal methods are a class of \acp{fem} that have gained popularity in recent years due to their ability to relax conformity constraints on meshes. This flexibility makes them well-suited for handling complex geometries. Moreover, the support for general polytopal elements facilitates adaptive mesh refinement and coarsening—key features for applications such as the modelling of geophysical flows \cite{dipietro-droniou-lectureNoteHHO-2025}, where the mesh may include elements that are unnecessarily small or large in certain regions. In addition to addressing limitations of traditional \ac{fem} from a mesh perspective, polytopal methods eliminate the need for predefined local spaces by reconstructing quantities of interest directly from degrees of freedom. This approach simplifies the extension to higher-order schemes. For more details, we refer to several popular polytopal \ac{fem}s: the \ac{hho} method \cite{di-pietro.droniou:2020:hybrid,Cicuttin-Ern-Pignet-2021-HHO-Book}, the \ac{hdg} method \cite{Cockburn-Gopalakrishnan-Lazarov-2009-HDG}, the \ac{vem} \cite{Veiga-Brezzi-Marini-Russo-2023-VEM}, the \ac{wg} method \cite{Wang-Ye-WG}, and the \ac{dg} method \cite{cangiani.dong.ea:2017:discontinuous}, and references therein.

The design of efficient, robust and scalable solvers for linear systems arising from these kinds of discretisations is important to make them competitive with traditional methods in real-world applications. These include, but are not limited to, \ac{gmg} and non-overlapping \ac{ddm} solvers. We mention for example \cite{Cockburn-Dubios-Gopalakrishnan-Tan-2014-Multigrid-HDG,DPMMR,Di_Pietro-Dong-Kanschat-Rupp-2024-Multigrid} as a few references on the design and analysis of GMG solvers for polytopal methods. In this work, we focus on non-overlapping \ac{ddm}-based solvers, which include \ac{bddc} \cite{Dohrmann-BDDC-Algorithm} and \ac{feti-dp} \cite{Farhat-Lesoinne-LeTallec-Pierson-Rixen-2001-FETI-DP}. These solvers rely on the exchange of information across inter-subdomain boundaries, and require three main ingredients: a {\it trace inequality}, which implies that the restriction of functions to the subdomain interface is stable; a {\it lifting result}, which lifts this restriction to the interior of the neighbouring subdomain; and continuity of a face truncation operator on piecewise polynomial functions in $H^1(\dom)$. The trace and lifting results facilitate the stable propagation of ``information'' globally between subdomains. In other words, trace theory bounds the injection operator from the space of discontinuous solutions across subdomain boundaries into the original space, defined via a weighting operator and a subdomain-local discrete harmonic extension. The bound on the face restriction operator leads to a mesh-dependent logarithmic estimate observed in \ac{bddc} and \ac{feti-dp} preconditioners. We note that the condition number of \ac{bddc} and \ac{feti-dp} preconditioned systems are identical.

For gradient-conforming \ac{fem}, this is realised with the help of continuous trace theory and standard piecewise polynomial truncation estimates \cite{Toselli-Widlund-DDM-book,Brenner.Scott:2008:MTFEM}. These preconditioners \cite{Dryja_original} and their analyses have been extended to other conforming polytopal discretisations, such as the conforming \ac{vem}~\cite{Bertoluzza-Pennacchio-Prada-2017-BDDC-VEM,Bertoluzza-Pennacchio-Prada-2020-FETI-VEM} and isogeometric analysis~\cite{BEIRODAVEIGA2013}. For non-conforming \ac{fem}, the trace inequality and lifting property are non-trivial, since the trace of piecewise polynomial functions in $L^2(\dom)$ do not have $H^{1/2}(\partial\dom)$-regularity, where $\dom$ is the polytopal domain of interest in $\mathbb{R}^d$ $(d \geq 2)$ with boundary $\partial\dom$. In the current state-of-the art, the analysis of solvers for hybridised formulations relies on the existence of an interpolant onto a conforming finite element space and then leveraging the continuous trace theory. This idea was originally proposed in \cite{Cowsar-Mandel-Wheeler-1995-BDDC-Mixed} for \ac{bnn} preconditioning of mixed \ac{fem} and further extended to other non-conforming discretisations, e.g., \ac{dg} \cite{Diosady-Darmofal-2012-BDDC-DG}, \ac{hdg} \cite{Tu-Wang-BDDC-2016-HDG}, and \ac{wg} \cite{Tu-Wang-2018-BDDC-WG}. However, since this approach hinges on the construction of a conforming interpolant, all the analyses so far have been limited to conforming simplicial or quadrilateral/hexahedral meshes. Furthermore, to the best of our knowledge, there is no literature available addressing non-conforming \ac{vem} or \ac{hho}.

In this work, we derive condition number bounds for \ac{bddc} and \ac{feti-dp} preconditioners for non-conforming polytopal hybrid discretisations on quasi-uniform general polytopal meshes (see Assumption \ref{assum:reg.mesh}). In order to do this, we rely on the discrete trace theory recently developed in \cite{Badia-Droniou-Tushar-2024-DiscreteTraceTheory} for non-conforming polytopal methods. Without the need of conforming interpolants, this reference designs a discrete $H^{1/2}(\partial\dom)$-seminorm and shows that it enjoys properties analogous to the continuous trace seminorm, namely discrete trace inequality and discrete lifting with respect to a discrete $H^1(\dom)$-seminorm. To be applicable to the analysis of \ac{ddm}, we extend the theory by analysing the continuity of face truncation operators. Additionally, we have implemented the proposed solvers in GridapSolvers.jl \cite{Manyer-2024-GridapSolvers}, a package that provides distributed linear solvers for Gridap.jl, a Julia library for the numerical approximation of \ac{pde} \cite{Verdugo-Badia-2022-Gridap,Badia-Verdugo-2020-Gridap}. Through extensive numerical experiments conducted on a supercomputer, we have demonstrated the robustness and weak scalability of our solvers, successfully scaling up to several hundreds of processors while maintaining optimal performance.

The remainder of this paper is organised as follows. We first recall the construction of the discrete trace seminorm and main results from \cite{Badia-Droniou-Tushar-2024-DiscreteTraceTheory} in Section \ref{sec:Prelims}. In Section \ref{sec:FaceRestrictionOp}, we extend the theory developed in \cite{Badia-Droniou-Tushar-2024-DiscreteTraceTheory} by analysing the continuity of the face truncation operator with respect to the discrete trace seminorm on the discrete $H^1(\dom)$-seminorm (see Section \ref{sec:FaceRestrictionOp} and Theorem \ref{thm:TruncationEst}). The design of the \ac{bddc} preconditioner based on hybrid spaces in a functional setting is described in Section \ref{sec:bddc.setting}, with the main result on robustness of the preconditioner with respect to mesh size stated in Theorem \ref{thm:bddc} and proved in Section \ref{sec:bddc.analysis}. In Section \ref{sec:applications}, we numerically verify the theoretical results. We design an experiment in Section \ref{sec:Truncation.NumExp} which verifies the truncation estimate of Theorem \ref{thm:TruncationEst}. We also present weak scalability tests (see Section \ref{sec:WeakScalabilityTests} and Section \ref{sec:jumping.coefficients}) for the proposed preconditioner on two classes of discontinuous skeletal methods, namely \ac{hdg} and \ac{hho}. A brief description of the tested methods is given in Section \ref{sec:disc.skel.methods}. The tests verify the theoretical condition number bounds of Theorem \ref{thm:bddc} and show robustness of the preconditioner with respect to mesh size and number of subdomains. 

\subsection{Preliminaries}\label{sec:Prelims}
We briefly recall here the setting and main results in \cite{Badia-Droniou-Tushar-2024-DiscreteTraceTheory}.
Let $\dom \subset \mathbb{R}^d$ be a bounded polytopal domain for $d \geq 2$. We consider a partition of $\dom$ into disjoint polytopes gathered in a set $\Th$, which we refer to as \emph{cells}. The mesh \emph{skeleton} is the union of the boundaries of the cells (composed of faces), and for each $t\in \Th$, $\mathcal{F}_t$ denotes the set of faces of $t$. We define the set of boundary faces as
\[
\Fhb \doteq \bigcup_{t \in \mathcal{T}_h} \{ f \in \mathcal{F}_t \ : \ f \subset \partial \Omega \},
\] 
and the set of interior faces as
\[
\Fhi \doteq \bigcup_{t\in\mathcal T_h}\{f\in\mathcal F_t\ :\ f\subset\Omega\}.
\] 
The union of these two sets forms the set $\Fh$ of all faces of the mesh.
The \emph{hybrid} mesh $\Mh \doteq (\Th, \Fh)$ (see \cite[Definition 1.4]{di-pietro.droniou:2020:hybrid}) is the combination of cells and faces.
For any cell or face $Z$, we denote by $h_Z$ its diameter, and we define the \emph{mesh size} as $h = \max_{t \in \Th} h_t$. We make the following assumption on the mesh.

\begin{assumption}[Mesh regularity]\label{assum:reg.mesh}
  The mesh $\Mh$ is regular as per \cite[Definition 1.9]{di-pietro.droniou:2020:hybrid}, and quasi-uniform: there exists $\rho>0$ independent of $h$ such that, for all $t\in\Th$, $h\le \rho h_t$. 
\end{assumption} 

This assumption ensures that for each $Z\in\Th\cup\Fh$ there exists a point $x_Z\in Z$ such that $Z$ contains a ball of radius $\varpi h_Z$, with $\varpi$ independent of $h$. From hereon, we fix such a \emph{centroid} for each cell and face; hence, in the following, $x_t$ or $x_f$ refer to these centroids, for a cell $t$ or a face $f$, respectively.

On the mesh, we define a \emph{hybrid} space $\Uh = \Uhi \times \Uhb$ of polynomials in the cells and on the faces.\footnote{The use of a wide hat is motivated by the introduction of additional spaces in the definition of the \ac{bddc} preconditioner in Section \ref{sec:bddc.setting}.} Given $k\ge 0$, we set
\[
\Uhi \doteq \bigtimes_{t \in \Th} \mathbb{P}_k(t), \quad
\Uhb \doteq \bigtimes_{f \in \Fh} \mathbb{P}_k(f),  
\] 
where $\mathbb{P}_k$ is the space of polynomials of degree at most $k$. An element in $\Uh$ will be denoted as $\ul{v}_h = ((v_t)_{t \in \Th}, (v_f)_{f \in \mathcal{F}_t})$, with $v_t\in \mathbb{P}_k(t)$ for all $t\in\Th$ and $v_f\in \mathbb{P}_k(f)$ for all $f\in\Fh$.

The discrete $H^1(\dom)$-seminorm of $\ul{v}_h \in \Uh$ is defined as
\begin{equation*}
  \seminorm{1,h}{\ul{v}_h}\doteq \left[ \sum_{t\in\Th}  \left(\norm{L^2(t)}{\nabla v_t}^2+\sum_{f\in\mathcal{F}_t}h_t^{-1}\norm{L^2(f)}{v_f-v_t}^2\right)\right]^{1/2}.
\end{equation*}

We introduce the boundary space $\Uhbb \doteq \bigtimes_{f \in \Fhb} \mathbb{P}_k(f)$ 
and define the trace operator $\tr:\Uh \to \Uhbb$ by
   \[
   \tr (\ul{v}_h)|_f = v_f\qquad\forall f\in\Fhb\,,\quad\forall \ul{v}_h\in\Uh.
   \]
The discrete $H^{1/2}(\partial \dom)$-seminorm on $\Uhbb$ is given by
\begin{equation}\label{eq:def.disc.Hhalf}
  \tnorm{1/2,h}{w_h}^2 \doteq \sum_{f\in\Fhb} h_f^{-1} \norm{L^2(f)}{w_f - \ol{w}_f}^2 + \sum_{(f,f')\in\FFhb} |f|_{d-1} |f'|_{d-1} \frac{|\ol{w}_f - \ol{w}_{f'}|^2}{\dffp^d},
\end{equation}
where
\begin{equation*}
  \ol{w}_f \doteq \frac{1}{|f|_{d-1}} \int_{f} w_f \quad\text{ and }\quad\FFhb=\{(f,f')\in\Fhb\times\Fhb\,:\,f\neq f'\}.
\end{equation*}

These discrete $H^1$-seminorm, trace operator, and $H^{1/2}$-seminorm satisfy the following trace and lifting properties, proved in \cite{Badia-Droniou-Tushar-2024-DiscreteTraceTheory}. Here and in the following, $a\lesssim b$ means that $a\le c b$ for some constant $c>0$ depending only on $\Omega$, the mesh regularity parameters (see Assumption \ref{assum:reg.mesh}) and the polynomial degree $k$. Equivalence $a\simeq b$ means that $a\lesssim b$ and $b\lesssim a$.

\begin{theorem}[Trace inequality]\label{thm:trace}
  The following discrete trace inequality holds:
  \begin{equation}\label{eq:trace}
    \tnorm{1/2,h}{\tr (\ul{v}_h)} \lesssim \seminorm{1,h}{\ul{v}_h}\qquad\forall \ul{v}_h\in\Uh.
  \end{equation}
\end{theorem}

\begin{theorem}[Lifting]\label{thm:lifting}
  There exists a discrete linear lifting operator $\mathcal{L}_h: \Uhbb \rightarrow \Uh$ such that: 
  \begin{equation}\label{eq:lifting}
    \seminorm{1,h}{\mathcal{L}_h(w_{h})}\lesssim \tnorm{1/2, h}{w_{h}}\qquad\forall w_h\in\Uhbb, \qquad \gamma \circ \mathcal{L}_h  = \text{Id}_{\Uhbb}.
  \end{equation}
\end{theorem}
We note that the $H^{1/2}$-seminorm in \eqref{eq:def.disc.Hhalf} has been introduced in \cite{Badia-Droniou-Tushar-2024-DiscreteTraceTheory} to prove the lifting property. Previous analyses using analogous techniques only considered $L^2(\partial\dom)$-norms \cite{Droniou.Eymard:2018:GDM} or only measuring local variations (first term in \eqref{eq:def.disc.Hhalf}) \cite[Eq.~(3.19)]{Cockburn-Dubios-Gopalakrishnan-Tan-2014-Multigrid-HDG}, which prevented the proof of a lifting property -- for which it is essential to control the long-range variations along the boundary, through the second term in the discrete $H^{1/2}$-seminorm.

\section{Truncation estimates for piecewise broken polynomial functions}\label{sec:FaceRestrictionOp}

In this section we establish an estimate in the discrete $H^{1/2}$-seminorm for truncated boundary functions, in the spirit of \cite[Equation 4.18]{Brenner-Park-Sung-2017-BDDC}, \cite[Lemma 4.1]{Bertoluzza-Pennacchio-Prada-2017-BDDC-VEM}, \cite[Section 7.5]{Brenner.Scott:2008:MTFEM} and \cite[Section 4.6]{Toselli-Widlund-DDM-book}. We consider a portion $\Gamma$ of $\partial\dom$ that satisfies three key properties: it is compatible with the underlying mesh, its size is comparable to that of $\Omega$, and its boundary has finite $(d-2)$-dimensional measure. These requirements are formalized in the following assumption.

\begin{assumption}[Regularity of $\Gamma$]\label{assum:reg.Gamma}~
    \begin{itemize}
    \item There is $\Fh(\Gamma)\subset\Fhb$ such that $\Gamma=\bigcup_{f\in\Fh(\Gamma)}\mathrm{cl}(f)$, where $\mathrm{cl}(f)$ is the closure of $f$.
    \item There is $C_\Gamma>0$ such that $C_\Gamma^{-1}\mathrm{diam}(\Omega)^{d-1}\le |\Gamma|_{d-1}$ and $|\partial\Gamma|_{d-2}\le C_\Gamma \mathrm{diam}(\Omega)^{d-2}$.
  \end{itemize}
\end{assumption}  

We then take $\ul{v}_h \in \Uh$ such that 
\begin{equation}\label{eq:int.F.v}
  \int_{\Gamma}\tr (\ul{v}_h) = 0
\end{equation}
and consider $w_h\in\Uhbb$ equal to $\tr (\ul{v}_h)$ on $\Gamma$ and to $0$ outside $\Gamma$:
\begin{align}\label{bddc.wf}
  w_f = 
  \left\{
  \begin{aligned}
    &v_f \;\; \mbox{if} \;\; f \in \Fh(\Gamma), \\
    &0  \;\;\;\; \mbox{if} \;\; f \notin \Fh(\Gamma),
  \end{aligned}
  \right.
  \quad \forall f\in\Fhb.
\end{align}

The following theorem provides an estimate for the truncated trace function $\tr(\ul{v}_h)$. To establish this result, we rely on the uniform cone condition satisfied by the polytopal domain $\Omega$. According to \cite[Section 4.8]{Adams-Fournier-Book}, this condition ensures the existence of: (i) a fixed reference cone $C$, (ii) a finite collection of subsets $((\partial \Omega)_i)_{i\in I}$ covering $\partial\Omega$ (with $\#I\sim 1$), and (iii) corresponding cones $(C_i)_{i\in I}$ such that each $C_i$ is congruent to $C$ and satisfies $(\partial \dom)_i+C_i\subset \dom$.

\begin{theorem}[Estimate of truncated boundary function]\label{thm:TruncationEst}
  Let $\ul{v}_h \in \Uh$ satisfy \eqref{eq:int.F.v} and $w_h \in \Uhbb$ be defined by \eqref{bddc.wf}. Under Assumption \ref{assum:reg.Gamma}, the following estimate holds:
  \begin{align}\label{bddc.main}
    \tnorm{1/2,h}{w_h} \lesssim \left(1 + \ln \left(\frac{\mathrm{diam}(\Omega)}{h}\right)\right) \seminorm{1,h}{\ul{v}_h},
  \end{align}
  where the hidden constant depends on the uniform cone condition parameters $(C,I,((\partial \Omega)_i)_{i\in I})$ for $\Omega$, the constant $C_\Gamma$ from Assumption \ref{assum:reg.Gamma}, the mesh regularity parameters from Assumption \ref{assum:reg.mesh}, and the polynomial degree $k$.
\end{theorem}

In the rest of this section, the hidden constants in the inequalities have the same dependencies as in this theorem.

\begin{remark}[Scaling]\label{rem:Truncation.Scaling}
  Without loss of generality, we assume that $\dom$ has unit diameter in our analysis. This assumption is justified by a scaling argument: when $\dom$ is scaled by a factor $\mu$, both the discrete $H^{1/2}(\partial\dom)$ and $H^1(\dom)$-seminorms scale with $\mu^{d-2}$. Consequently, the hidden constants remain independent of $\mbox{diam}(\dom)$. Additionally, the constant $C_\Gamma$ from Assumption \ref{assum:reg.Gamma} remains invariant under domain scaling.
\end{remark}

\subsection{Notations}\label{sec:TruncEst.Notations}
Building upon the framework established in \cite{Badia-Droniou-Tushar-2024-DiscreteTraceTheory}, we first recall the concept of distance measurements along $\partial\dom$ that was instrumental in proving Theorems \ref{thm:trace} and \ref{thm:lifting} for polytopal domains. To establish Theorem \ref{thm:TruncationEst}, we will need to work with carefully constructed partitions of both the face and cell sets. In what follows, we introduce these partitions along with the necessary notation. A summary of the notations used in this section is provided in Tables \ref{Table:Notation_Symbols} and \ref{Table:Notation_Sets}.

Given $l \geq 1$, we define the set $\W_l$ as the collection of face pairs $(f,f') \in \FFhb$ whose centroids are {\it within a distance $lh$ of each other} in the set, i.e.,
\begin{align}\label{set:Wl}
  \W_l \doteq \{(f,f')\in\FFhb\,:  (l-1)h \leq |x_f-x_{f'}| < lh\}.
\end{align}
For each face $f \in \Fhb$, we define its corresponding ``annulus of faces'' $\W_{lf}$,
\begin{align}\label{set:Wlf}
\W_{lf} \doteq \{ f' \in \Fhb : (f,f') \in \W_l \}.
\end{align} 

For any face $f \in \Fh(\Gamma)$, we define $\dfGbd$ as the distance along $\partial\dom$ from its centroid $x_f$ to the boundary of $\Gamma$. We partition the faces in $\Fh(\Gamma)$ into layers based on their distance from the boundary, defining for each $r \geq 0$ the layer
\begin{align}\label{def:bddc.layers}
  \G_r \doteq \{f \in \Fh(\Gamma): rh \leq \dfGbd < (r+1)h\}.
\end{align}
Figure \ref{fig:bddc-1} provides a visual representation of this layering. While the figures depict $\Gamma$ as a flat face of $\Omega$ for clarity, this is not required in our subsequent analysis.
  
\begin{figure}[htbp!]
  \includegraphics[width=.7\linewidth]{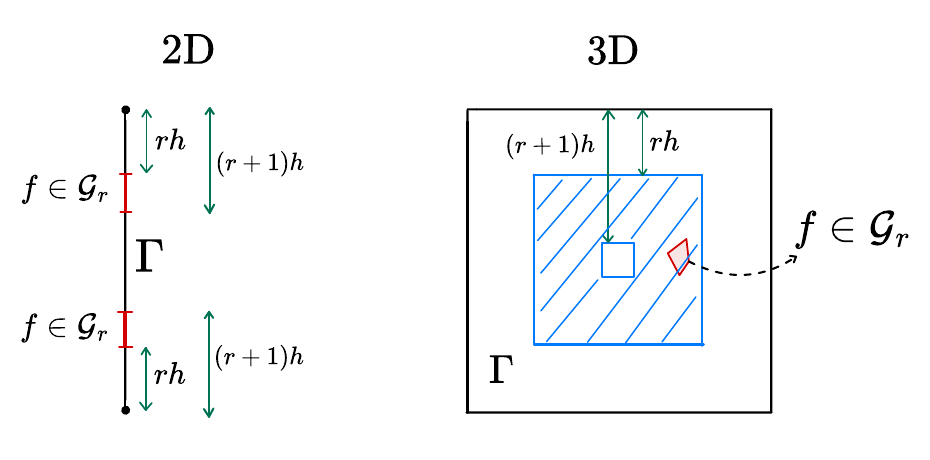}
  \caption{Layers of faces from the boundary of $\Gamma$.}
  \label{fig:bddc-1}
\end{figure}

Let $d_0$ denote the height of the reference cone $C$ in the uniform cone condition, which satisfies $d_0\simeq\mathrm{diam}(\Omega)=1$. For each layer $\G_r$, we decompose it according to the covering $((\partial \Omega)_i)_{i\in I}$ by writing $\G_r = \bigcup_{i \in I} \G_r^i$ where $\G_r^i=\G_r\cap(\partial\Omega)_i$. By construction, each subset $\G_r^i$ satisfies $\G_r^i+C_i\subset \dom$ for all $i\in I$ (see Figure \ref{fig:bddc-2}). In our subsequent analysis, we need to bound sums of the form $\sum_{f\in\G_r}a_f$ where $a_f\ge 0$. Since the number of covering sets $I$ is bounded independently of $h$ ($\#I\sim 1$), it suffices to bound $\sum_{f\in\G_r^i}a_f$ for each $i\in I$. For clarity of presentation, we will focus on a single subset $\G_r^i$ and omit the index $i$ in what follows.

\begin{figure}[htbp!]
  \includegraphics[width=.7\linewidth]{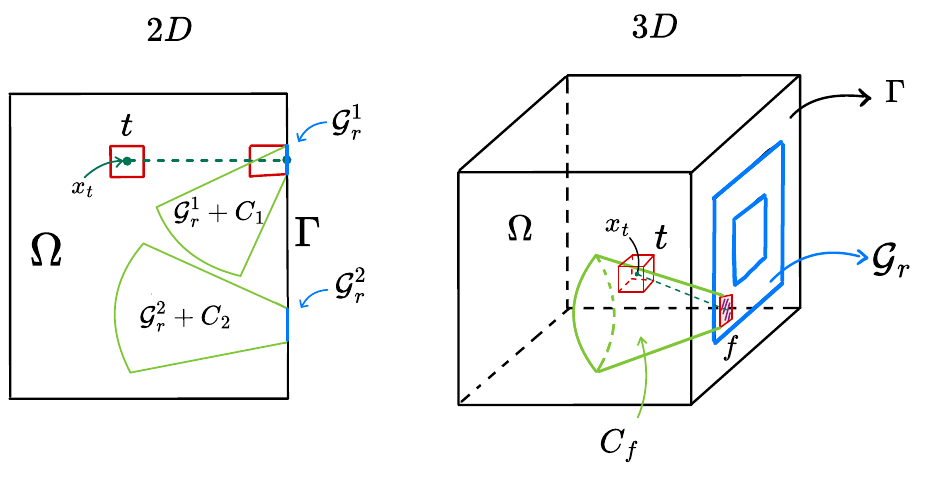}
  \caption{Illustration of the uniform cone condition.}
  \label{fig:bddc-2}
\end{figure}

For any face $f\in\G_r$, we define $\Co_f$ as the collection of cells that intersect the cone $x_f + C$:
\begin{equation}\label{def:bddc.Cf}
  \Co_f \doteq \{t \in \Th: t \; \mbox{intersects} \; x_f + C\}.
\end{equation}
We further introduce $\Co_{f{d_0}}$ as the subset of cells in $\Co_f$ that lie at a distance of at least $\frac{d_0}{2} - h$ from $x_f$:
\begin{equation}\label{eq:def.Cfd0}
  \Co_{f{d_0}} \doteq \left\{t \in \Co_f:  |x_t - x_f| \geq \frac{d_0}{2} - h\right\}.
\end{equation}
This set $\Co_{f{d_0}}$ represents the cells that intersect only the ``second half" of the cone $\Co_f$, away from the face $f$ (see Figure \ref{fig:bddc-3}).

\begin{figure}[htbp!]
  \includegraphics[width=7cm,height=6cm]{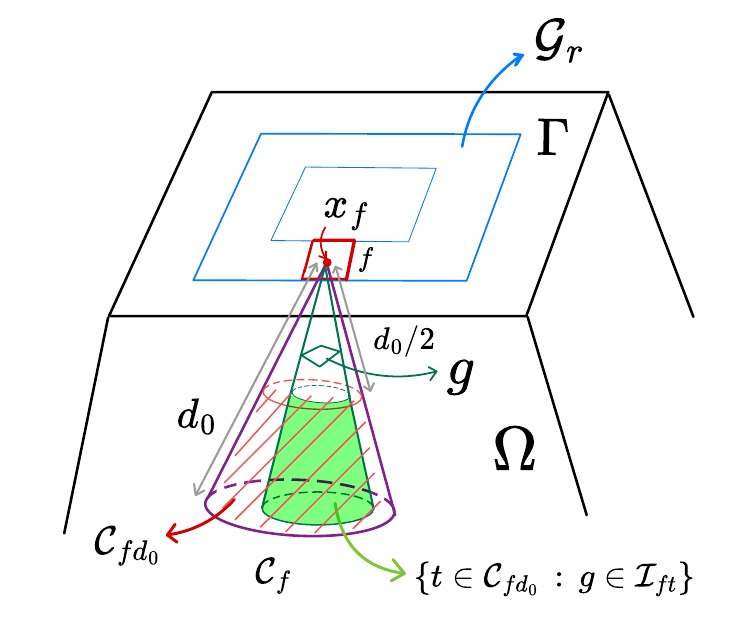}
  \caption{Visual representation of the sets $\Co_f$ (full cone), $\Co_{fd_0}$ (truncated cone), and the set discussed in Lemma \ref{lem:bddc.Cfd}-\ref{lem:bddc.Cfd.card}. Note that while these sets are unions of cells, the figure shows a simplified geometric approximation of the domain they cover.}
  \label{fig:bddc-3}
\end{figure}
For any face $f \in \Fh(\Gamma)$ and cell $t \in \Th$, we define $\Ift$ as the set of internal faces that intersect the line segment $[x_f, x_t]$ between their centroids:
\begin{align}\label{def:set.Ift}
  \Ift \doteq \{g \in \Fhi: g \; \mbox{intersects} \; [x_f, x_t]\}.
\end{align}
We partition $\Ift$ into concentric layers based on the distance from each face to $f$. For each non-negative integer $m$, the $m$-th layer consists of faces at a distance between $mh$ and $(m+1)h$ from $f$:
\begin{align}\label{def:set.Iftm}
  \Ift^m = \{g \in \Ift: mh \leq \dfg < (m+1)h\}.
\end{align}

\begin{table}[htbp!]
  \renewcommand*{\arraystretch}{1.4}
  \begin{tabular}{@{}l@{\hspace{1em}}p{0.8\textwidth}@{}}
    \toprule
    \textbf{Symbol} & \textbf{Description} \\
    \midrule
    $x_Z$ & Centroid of $Z$, a face or cell of the mesh \\
    \midrule
    $\dfGbd$ & Distance between $x_f \in \Fh(\Gamma)$ and the boundary of $\Gamma \subset \partial\Omega$, taken along $\partial\dom$ \\
    \midrule
    $D(x,r)$ & Disc on $\partial\dom$ centered at $x$ with radius $r$ \\
    \bottomrule
  \end{tabular}
  \caption{Points, faces and distance}
  \label{Table:Notation_Symbols}
\end{table}

\begin{table}[htbp!]
  \renewcommand*{\arraystretch}{1.4}
  \begin{tabular}{@{}l@{\hspace{1em}}p{0.8\textwidth}@{}}
    \toprule
    \textbf{Symbol} & \textbf{Description} \\
    \midrule
    $\FFhb$ & Set of pairs $(f,f')$ of distinct faces on $\partial\Omega$ \\
    \midrule
    $\W_l$ & Set of pairs of boundary faces that are approximately at distance $\simeq lh$ of each other. See \eqref{set:Wl} \\
    \midrule
    $\W_{lf}$ & Slice of the set $\W_l$ at the face $f$, that is, faces $f'$ that are \emph{approximately} at distance $lh$ of $f$. See \eqref{set:Wlf} \\
    \midrule
    $\G_r$ & Split of faces in $\Fh(\Gamma)$ according to their layers from the boundary of $\Gamma$. See \eqref{def:bddc.layers} \\
    \midrule
    $\Ift$ & For a fixed $f \in \Fh(\Gamma)$ and $t \in \Th$, the set of faces in $\dom$ that intersect the line segment $[x_f, x_t]$. See \eqref{def:set.Ift} \\
    \midrule
    $\Ift^m$ & Slice of $\Ift$ according to the distance to $f$. See \eqref{def:set.Iftm} \\
    \bottomrule
  \end{tabular}
  \caption{Sets of faces and regions}
  \label{Table:Notation_Sets}
\end{table}

\subsection{Proof of Theorem \ref{thm:TruncationEst}}\label{sec:Pf.FaceRestrictionOp}

Let $\FFh(\Gamma)$ denote the set of all distinct face pairs $(f,f')$ where both faces lie in $\Gamma$. The discrete $H^{1/2}$-seminorm \eqref{eq:def.disc.Hhalf} of $w_h$ (defined from $\ul{v}_h$ via \eqref{bddc.wf}) can be decomposed into four terms:
\begin{align*}
  \tnorm{1/2, h}{w_h}^2 
  &= \sum_{f\in \Fh(\Gamma)} h_f^{-1} \norm{L^2(f)}{v_f - \ol{v}_f}^2 + \sum_{(f,f')\in \FFh(\Gamma)} |f|_{d-1} |f'|_{d-1} \frac{|\ol{v}_f - \ol{v}_{f'}|^2}{\dffp^d} \\
  &\quad + \sum_{f \in \Fh(\Gamma)} \sum_{f' \in \Fhb \backslash \Fh(\Gamma)} |f|_{d-1} |f'|_{d-1} \frac{|\ol{v}_f|^2}{\dffp^d} + \sum_{f \in \Fhb \backslash \Fh(\Gamma)} \sum_{f' \in \Fh(\Gamma)} |f|_{d-1} |f'|_{d-1} \frac{|\ol{v}_{f'}|^2}{\dffp^d}.
\end{align*}
The first two terms represent the seminorm contributions from faces within $\Gamma$, while the last two terms account for interactions between faces in $\Gamma$ and faces outside $\Gamma$. The first two terms are bounded by $\tnorm{1/2, h}{\tr (\ul{v}_h)}^2$, which includes these terms summed over all faces. The last two terms are identical due to the symmetric roles of $f$ and $f'$. This leads to:
\begin{equation}\label{eq:bddc.Hhalf}
  \tnorm{1/2, h}{w_h}^2 \leq \tnorm{1/2, h}{\tr (\ul{v}_h)}^2 + 2 \underbrace{\sum_{f \in \Fh(\Gamma)} \sum_{f' \in \Fhb \backslash \Fh(\Gamma)} |f|_{d-1} |f'|_{d-1} \frac{|\ol{v}_f|^2}{\dffp^d}}_{\doteq E_0} \overset{\eqref{eq:trace}}\lesssim \seminorm{1, h}{\ul{v}_h}^2 + E_0.
\end{equation}
Given that $\mathrm{diam}(\Omega)=1$ (as established in Remark \ref{rem:Truncation.Scaling}), the proof of \eqref{bddc.main} hinges on demonstrating the following bound:
\begin{align}\label{bddc.E0}
  E_0 \lesssim (1 + |\ln h|)^2 \seminorm{1,h}{\ul{v}_h}^2.
\end{align}

We decompose the sum over ${f' \in \Fhb \backslash \Fh(\Gamma)}$ in $E_0$ into concentric annuli around each face $f\in\Fh(\Gamma)$ using the sets $\W_l$ and $\W_{lf}$ from \eqref{set:Wl} and \eqref{set:Wlf}:
\begin{equation}\label{eq:part.l0.1}
  \sum_{f' \in \Fhb \backslash\Fh(\Gamma)} \frac{|f'|_{d-1}}{\dffp^d} = \sum_{l=l_0}^{L} \sum_{f' \in \W_{lf} \backslash \Fh(\Gamma)} \frac{|f'|_{d-1}}{\dffp^d},
\end{equation}
where $L$ is chosen so that $Lh \ge\mathrm{diam}(\Omega)= 1$, ensuring that $(\W_{lf})_{l = 1, \dots, L}$ forms a complete partition of $\Fhb$ (the choice of $l_0\ge 1$ will be discussed below). For faces $f'\in \W_{lf}$, the distance $\dffp$ satisfies $(l-1)h\le \dffp <lh$. When $l\ge 2$, the distances $(l-1)h$, $lh$, and $(l+1)h$ are comparable, yielding $\dffp\simeq (l+1)h$. This scaling also holds for $l=1$ since mesh regularity and $f'\neq f$ ensure $\dffp\gtrsim h$. Substituting $\dffp\simeq (l+1)h$ into \eqref{eq:part.l0.1} gives
\begin{equation}\label{eq:part.l0}
  \sum_{f' \in \Fhb \backslash\Fh(\Gamma)} \frac{|f'|_{d-1}}{\dffp^d} \simeq \sum_{l = l_0}^{L} \frac{1}{((l+1)h)^d} \sum_{f' \in \W_{lf}} |f'|_{d-1}.
\end{equation}
By \cite[Lemma 3.2]{Badia-Droniou-Tushar-2024-DiscreteTraceTheory}, the cardinality of $\W_{l f}$ is bounded by $\# \W_{l f} \lesssim l^{d-2}$ for all $f \in \Fhb$ and $l\ge 1$. Together with the mesh quasi-uniformity bound $|f'|_{d-1} \simeq h^{d-1}$, this leads to
\begin{equation}\label{eq:part.l0.2}
  \sum_{f' \in \Fhb \backslash \Fh(\Gamma)} \frac{|f'|_{d-1}}{\dffp^d} \lesssim \sum_{l=l_0}^{L} \frac{l^{d-2} h^{d-1}}{((l+1)h)^d} \le \sum_{l=l_0}^{L} \frac{h}{((l+1)h)^2}.
\end{equation}
To bound this sum, we observe that $\frac{1}{((l+1)h)^2}\le \frac{1}{x^2}$ for $x\in [lh,(l+1)h]$, which implies
\begin{equation}\label{eq:part.l0.3}
  \sum_{l = l_0}^{L}  \frac{h}{((l+1)h)^2} \leq \sum_{l = l_0}^{L} \int_{l h}^{(l+1)h} \frac{dx}{x^2} = \int_{l_0 h}^{(L+1)h} \frac{dx}{x^2} \leq  \frac{1}{l_0 h}.
\end{equation}
The value of $l_0$ is determined by the requirement that $\W_{lf}\backslash\Fh(\Gamma)\neq\emptyset$ for \eqref{eq:part.l0} to be valid. This ensures the existence of a face $f' \in \Fhb \backslash \Fh(\Gamma)$ with $\dffp \le l_0 h$, establishing that 
$$
  \dfGbd = \dist(x_f,\partial \Omega\backslash \Gamma) \le |x_f - x_{f'}| \le l_0h
$$
(where distances are measured along $\partial \Omega$). Using this $l_0$ in \eqref{eq:part.l0.3} and combining with \eqref{eq:part.l0.2} yields
$$
  \sum_{f' \in \Fhb \backslash \Fh(\Gamma)}  \frac{|f'|_{d-1}}{\dffp^d} \lesssim \frac{1}{\dfGbd}, 
$$
and consequently
\begin{equation}\label{eq:est.E0}
  E_0 \lesssim \sum_{f \in \Fh(\Gamma)} |f|_{d-1} \frac{|\ol{v}_f|^2}{\dfGbd}.
\end{equation}
We now organize the faces in this sum according to the layers $\G_r$ from \eqref{def:bddc.layers}. Taking $R$ such that $Rh\simeq \mathrm{diam}(\Omega)=1$ ensures that $(\G_r)_{r=1,\dots,R}$ forms a complete partition of $\Fh(\Gamma)$. This gives
\begin{equation*}
  E_0 \lesssim \sum_{r=1}^{R} \sum_{f \in \G_r} |f|_{d-1} \frac{|\ol{v}_f|^2}{\dfGbd} \leq \sum_{r=1}^{R} \frac{1}{rh} \sum_{f \in \G_r} |f|_{d-1} |\ol{v}_f|^2.
\end{equation*}
Applying Proposition \ref{prop:bddc.L2} leads to
\begin{equation}\label{eq:bddc.E0.2}
  E_0 \lesssim \left(\sum_{r=1}^{R} \frac{h}{rh}\right) (1 + |\ln h|) \seminorm{1,h}{\ul{v}_h}^2. 
\end{equation}
We bound the remaining sum by writing
\begin{align}\label{eq:bddc.log}
  \sum_{r=1}^{R} \frac{h}{rh} = 1 + \sum_{r=2}^{R} \frac{h}{rh} \leq 1 + \sum_{r=2}^{R} \int_{(r-1)h}^{rh} \frac{ds}{s} = 1 + \int_{h}^{Rh} \frac{ds}{s} = 1 + \ln R.
\end{align}
Since $Rh\simeq 1$, we have $\sum_{r=1}^{R} \frac{h}{rh} \lesssim 1 + |\ln h|$. Substituting this into \eqref{eq:bddc.E0.2} proves \eqref{bddc.E0}. Finally, combining \eqref{bddc.E0} with \eqref{eq:bddc.Hhalf} establishes the desired estimate \eqref{bddc.main}. \qed

Before proving Proposition \ref{prop:bddc.L2} - a key ingredient in the preceding analysis - we first establish several technical lemmas that will streamline its demonstration.

\begin{proposition}\label{prop:bddc.L2}
  Let $\ul{v}_h\in\Uh$ satisfy \eqref{eq:int.F.v}. For any layer index $r\ge 1$, the following bound holds for the weighted sum of squared face averages over faces in layer $\G_r$:
  \begin{equation}
    \sum_{f \in \G_r} |f|_{d-1} |\ol{v}_f|^2 \lesssim h (1 + |\ln h|) \seminorm{1,h}{\ul{v}_h}^2.
  \end{equation}
\end{proposition}

Throughout this section, we denote by $D(x,r)$ the disc centered at point $x$ with radius $r$ that lies on the boundary $\partial\dom$.

\begin{lemma}[Cardinality bound for $\Ift^m$]\label{lem:bddc.Iftm}
  For any boundary face $f \in \Fh(\Gamma)$, cell $t\in\Th$, and layer index $m\ge 0$, the number of faces in the set $\Ift^m$ is bounded by a constant independent of the mesh size, i.e., $\#\Ift^m \lesssim 1$.
\end{lemma}
\begin{proof}
  Let $g \in \Ift$ be a face (see \eqref{def:set.Ift}). By definition, $g$ lies within a cylinder with base $D(x_f, h)$ in the plane orthogonal to $[x_f,x_t]$ and axis along $(x_f, x_t)$. For faces $g\in\Ift^m$ (see \eqref{def:set.Iftm}), we have the additional constraint that $g$ must intersect the cylinder section between heights $mh - h$ and $(m+1)h+h$.
  Each face $g\in\Ift^m$ belongs to a cell $t'\in\Th$ that lies entirely within the $h$-enlargement of this cylinder section. This enlarged section has base $D(x_f,2h)$ and extends from height $(m-2)h$ to $(m+3)h$ along $(x_f,x_t)$, giving a total measure of $|D(x_f,2h)|_{d-1}\times 5h\lesssim h^d$.
  By mesh quasi-uniformity, each cell $t'$ has measure $\simeq h^d$, so the number of such cells is $\lesssim h^d/h^d=1$. Mesh regularity further ensures that each cell $t'$ contains $\lesssim 1$ faces $g$. Therefore, $\#\Ift^m\lesssim 1$ as desired.
\end{proof}

\begin{lemma}[Volume and Cardinality Estimates for $\Co_{f d_0}$]\label{lem:bddc.Cfd}
    Let $f \in \G_r$ be fixed. Then:
    \begin{enumerate}[label=(\roman*)]
      \item \label{lem:bddc.Cfd.sum} The total volume of cells in $\Co_{f{d_0}}$ is bounded below by a constant independent of $h$: $\sum_{t \in \Co_{f{d_0}}} |t|_{d} \gtrsim 1$.
      \item \label{lem:bddc.Cfd.card} For any interior face $g\in\Fhi$, the number of cells in $\Co_{f{d_0}}$ that contain $g$ in their interface set $\Ift$ is bounded above by: $\# (\{t \in \Co_{f{d_0}}: g \in \Ift\}) \lesssim \frac{1}{\dfg^{d-1} h}$. 
    \end{enumerate}
 \end{lemma}

\begin{proof}
  \textit{Proof of \ref{lem:bddc.Cfd.sum}}: Consider the cone $x_f+C$ and its section at distance $\ge d_0/2$ from $x_f$, where $d_0\simeq 1$. This section has measure $\gtrsim 1$, with the hidden constant depending on the cone's opening angle. The construction of $\Co_{fd_0}$ in \eqref{eq:def.Cfd0} includes a $-h$ offset that ensures $\bigcup_{t\in\Co_{fd_0}}t$ fully covers this section. Hence, $\sum_{t\in\Co_{fd_0}}|t|_d =|\bigcup_{t\in\Co_{fd_0}}t|_d \gtrsim 1$.
  
  \medskip
  \textit{Proof of \ref{lem:bddc.Cfd.card}}: Let $t \in \Co_{f{d_0}}$ be any cell. By definition, $t$ must intersect the portion of cone $x_f+C$ between heights $d_0/2-h$ and $d_0$. For a face $g \in \Ift$, we can restrict our attention to the portion of the cone spanned by the projection of $g$ onto the sphere $S(x_f,d_0/2)$ centered at $x_f$ with radius $d_0/2$ (see Figure \ref{fig:bddc-3}). Through homothety -- analogous to the case in \cite[Figure 3]{Badia-Droniou-Tushar-2024-DiscreteTraceTheory} -- this projection has diameter $\lesssim \frac{(d_0/2)}{\dfg}\times h_{g}$. 
  It follows that $t$ must lie within the $h$-enlargement of this cone portion, which is bounded by heights $d_0/2-h$ and $d_0$ and a disc of diameter $\lesssim \frac{(d_0/2)}{\dfg}h_{g}$ on $S(x_f,d_0/2)$.   This enlarged portion has measure $\lesssim (\frac{(d_0/2)}{\dfg}h)^{d-1}\times (\frac{d_0}{2}+h)\simeq \frac{h^{d-1}}{\dfg^{d-1}}$. By mesh quasi-uniformity, we conclude:
  $$
  \# (\{t \in \Co_{f{d_0}}: g \in \Ift\}) \lesssim \frac{h^{d-1}}{\dfg^{d-1}h^d} \simeq \frac{1}{\dfg^{d-1}h}. 
  $$
\end{proof}

\begin{lemma}[Weighted sum bound for interior faces]\label{lem:bddc.Q} 
    For any interior face $g\in\Fhi$, the following weighted sum over boundary faces and their associated cells is bounded by a constant independent of the mesh size:
    \begin{equation}\label{eq:bddc.Q}
    \sum_{f \in \G_r} \sum_{\substack{t \in \Co_{f{d_0}} \\ \text{s.t.} g \in \Ift}} |f|_{d-1} \dfg \lesssim 1.
    \end{equation}
\end{lemma}
\begin{proof}
  Let $J_{rg}= \{f\in\G_r\,:\,\exists t\in\Co_{f{d_0}}\text{ s.t. }g\in\Ift\}$. By Lemma \ref{lem:bddc.Cfd}-\ref{lem:bddc.Cfd.card}, we have
  \begin{align*}
    \sum_{f \in \G_r} \sum_{\substack{t \in \Co_{f{d_0}} \\ \text{s.t.} g \in \Ift}} |f|_{d-1} \dfg ={}&
    \sum_{f \in J_{rg}} |f|_{d-1} \dfg \times \# (\{t \in \Co_{f{d_0}}: g \in \Ift\}) \\
    \lesssim{}& \frac{1}{h}\sum_{f \in J_{rg}} \frac{1}{\dfg^{d-2}}|f|_{d-1}.
  \end{align*}
  Since $\dfg \ge \dist(x_{g},\Gamma)$ for all $f\in\Fh(\Gamma)$, we obtain
  \begin{equation}\label{eq:the.day.is.saved}
    \sum_{f \in \G_r} \sum_{\substack{t \in \Co_{f{d_0}} \\ \text{s.t.} g \in \Ift}} |f|_{d-1} \dfg \lesssim
    \frac{1}{\dist(x_{g},\Gamma)^{d-2}h}\sum_{f \in J_{rg}} |f|_{d-1}.
  \end{equation}
  We now characterize the set $J_{rg}$. For any $f\in J_{rg}$, there exists a cell $t$ intersecting the cone $x_f+C$ with $g \in \Ift$. This means $g$ intersects the $h$-enlargement of this cone, which in turn implies that $f$ must intersect the $2h$-enlargement of the reversed cone $x_{g}-C$. Therefore, $f$ lies within $\Gamma_{rg}$, defined as the intersection of $\Gamma$ with the $3h$-enlargement of $x_{g}-C$. By mesh regularity and $g\in\Fhi$, we have $\dist(x_{g},\Gamma)\gtrsim h$, so $\Gamma_{rg}$ is contained in a disc on $\Gamma$ of radius $\lesssim \dist(x_{g},\Gamma)$. 
  Let $G_r$ denote the region covered by faces in $\G_r$. By construction, $G_r$ forms a band of width $\lesssim h$ that follows $\partial\Gamma$ at a distance $\simeq rh$. Since $|\partial\Gamma|_{d-2}\lesssim 1$ by Assumption \ref{assum:reg.Gamma}, any $f\in J_{rg}$ must belong to $\Gamma_{rg}\cap G_r$.
  
  The intersection $\Gamma_{rg}\cap G_r$ is contained within the overlap of a disc of radius $\lesssim \dist(x_{g},\Gamma)$ and a band of width $\lesssim h$ following $\partial\Gamma$. This geometric constraint implies $|\Gamma_{rg}\cap G_r|_{d-1}\lesssim \dist(x_{g},\Gamma)^{d-2}h$, and thus
  $$
  \sum_{f \in J_{rg}} |f|_{d-1}\le |\Gamma_{rg}\cap G_r|_{d-1}\lesssim \dist(x_{g},\Gamma)^{d-2}h.
  $$
  Substituting this bound into \eqref{eq:the.day.is.saved} completes the proof.
\end{proof}

We are now ready to prove Proposition \ref{prop:bddc.L2}.

\begin{proof}[Proof of Proposition \ref{prop:bddc.L2}]
  Let $t \in \Co_f$ (i.e., $t$ intersects $x_f+C$) and recall that $\Ift$ denotes the set of internal faces $g$ intersected by $[x_f,x_t]$. We can express $\ol{v}_f-\ol{v}_t$ as a telescoping sum over these faces $g$ of the differences $(\ol{v}_{t_1} - \ol{v}_{g})+(\ol{v}_{g} - \ol{v}_{t_2})$, where $t_1,t_2$ are the cells adjacent to $g$ (with only one difference appearing for $g=f$), and $\ol{v}_{t'}$ represents the average of $v_{t'}$ over $t'\in\Th$. This yields
  \begin{align*}
    |\ol{v}_f| \leq |\ol{v}_t| + \sum_{g \in \Ift} \Delta_{g} \ol{v}
  \end{align*}
  where $\Delta_{g} \ol{v} = |\ol{v}_{t_1} - \ol{v}_{g}| + |\ol{v}_{g} - \ol{v}_{t_2}|$. Applying the Cauchy-Schwarz inequality with carefully chosen weights gives
  \begin{align}\label{eq:bddc.vf}
    |\ol{v}_f|^2 &\lesssim |\ol{v}_t|^2 + \left(\sum_{g \in \Ift} \Big[\dfg^{1/2}  h^{-1/2}  |g|_{d-1}^{1/2}\; \Delta_{g} \ol{v}\Big] \times \Big[\dfg^{-1/2} h^{1/2} |g|_{d-1}^{-1/2}\Big]\right)^2 \nonumber \\
    &\lesssim |\ol{v}_t|^2 + \left(\sum_{g \in \Ift} \dfg  h^{-1} |g|_{d-1} (\Delta_{g} \ol{v})^2\right) \times \left(\sum_{g \in \Ift} \dfg^{-1}  h  |g|_{d-1}^{-1}\right).
  \end{align}
  By mesh quasi-uniformity (Assumption \ref{assum:reg.mesh}), $|g|_{d-1} \simeq h^{d-1}$, so
  \begin{align*}
    T_1 \doteq \sum_{g \in \Ift} \dfg^{-1}  h |g|_{d-1}^{-1} \lesssim \sum_{g \in \Ift} \dfg^{-1}\; h^{2-d}.
  \end{align*}
  We partition this sum into slices $\Ift^m$ as defined in \eqref{def:set.Iftm} and let $R\simeq 1/h$ such that $(\Ift^m)_{m=0,\ldots,R-1}$ covers $\Ift$. This gives
  \begin{align*}
    T_1 &\lesssim h^{2-d}\sum_{m=0}^{R-1} \sum_{g \in \Ift^m} \frac{1}{\dfg} 
    \lesssim \#\Ift^0 \,h^{2-d} \times h^{-1} + \sum_{m=1}^{R-1} \#\Ift^m\, \frac{h^{2-d}}{mh},
  \end{align*}
  where we used $\dfg \gtrsim h$ for the first term and the definition of $\Ift^m$ for the second term. By Lemma \ref{lem:bddc.Iftm}, we obtain
  \begin{align*}
    T_1 \lesssim h^{1-d}\left(1 + \sum_{m=1}^{R-1} \frac{h}{mh}\right).
  \end{align*}
  The sum can be bounded using \eqref{eq:bddc.log} as $1 + \sum_{m=1}^{R-1} \frac{h}{mh}\le 1 + \ln R$. Since $R\simeq 1/h$, this yields
  \begin{align*}
    T_1 \lesssim h^{1-d} (1 + |\ln h|).
  \end{align*} 
  Substituting this bound into \eqref{eq:bddc.vf} leads to
  \begin{align*}
    |\ol{v}_f|^2 \lesssim |\ol{v}_t|^2 + h^{1-d} (1+|\ln h|) \sum_{g \in \Ift} \dfg h^{-1} |g|_{d-1} (\Delta_{g} \ol{v})^2.
  \end{align*}
  We multiply by $|t|_d \simeq h^d$ and sum over $t\in \Co_{f{d_0}}$. Applying \ref{lem:bddc.Cfd.sum} from Lemma \ref{lem:bddc.Cfd} gives
  \begin{align*}
    |\ol{v}_f|^2 &\lesssim \big(\sum_{t \in \Co_{f{d_0}}} |t|_{d}\big) |\ol{v}_f|^2 \nonumber\\
    &\lesssim \sum_{t \in \Co_{f{d_0}}} |t|_d |\ol{v}_t|^2 + h (1+|\ln h|) \sum_{t \in \Co_{f{d_0}}} \sum_{g \in \Ift} \dfg h^{-1} |g|_{d-1} (\Delta_{g} \ol{v})^2 \nonumber\\
    &\lesssim \norm{L^2(\dom)}{\widehat{v}_h}^2 + h (1+|\ln h|) \sum_{t \in \Co_{f{d_0}}} \sum_{g \in \Ift} \dfg h^{-1} |g|_{d-1} (\Delta_{g} \ol{v})^2,
  \end{align*}
  where $\widehat{\ul{v}}_h=((\ol{v}_t)_{t\in\Th}, (\ol{v}_f)_{f\in\Fh})$ denotes the projection of $\ul{v}_h$ onto the lowest-order hybrid space on $\Th$ (obtained by projecting each component onto constant polynomials), and $\widehat{v}_h$ is the piecewise constant function on $\dom$ defined from $(\ol{v}_t)_{t\in\Th}$. Multiplying by $|f|_{d-1}$ and summing over $f \in \G_r$ yields
  \begin{align}\label{eq:bddc.vf.1}
    \sum_{f \in \G_r} |f|_{d-1} |\ol{v}_f|^2 \lesssim h \norm{L^2(\Omega)}{\widehat{v}_h}^2 + h (1 + |\ln h|)  \underbrace{\sum_{f \in \G_r} |f|_{d-1} \sum_{t \in \Co_{f{d_0}}} \sum_{g \in \Ift} \dfg h^{-1} |g|_{d-1} (\Delta_{g} \ol{v})^2}_{\doteq T_2},
  \end{align}
  where the factor $h$ in the first term comes from the fact that faces in $\G_r$ lie within a band of width $\lesssim h$ following $\partial \Gamma$ at a distance $\simeq rh$ from this boundary (see Figure \ref{fig:bddc-1}).
  
  To estimate $T_2$, we first reorder the summations over faces, cells, and internal faces:
  \begin{align*}
    T_2 &= \sum_{g \in \Fhi} h^{-1} |g|_{d-1} (\Delta_{g} \ol{v})^2 \sum_{f \in \G_r} \sum_{\substack{t \in \Co_{f{d_0}} \\ \text{s.t.} g \in \Ift}} |f|_{d-1} \dfg\\
    \overset{\eqref{eq:bddc.Q}}&\lesssim \sum_{g \in \Fhi} h^{-1} |g|_{d-1} (\Delta_{g} \ol{v})^2 \doteq \seminorm{\dom,2}{\widehat{\ul{v}}_h}^2.
  \end{align*}
  Inserting this estimate into \eqref{eq:bddc.vf.1} and applying the discrete Poincaré inequality from \cite[Lemma B.22]{Droniou.Eymard:2018:GDM} (noting that the discrete $H^1$-seminorm in this reference is equivalent to $\seminorm{\dom,2}{{\cdot}}^2$ under mesh quasi-uniformity), we obtain
  \begin{align*}
    \sum_{f \in \G_r} |f|_{d-1} |\ol{v}_f|^2 &\lesssim h \norm{L^2(\dom)}{\widehat{v}_h}^2 + h (1 + |\ln h|) \seminorm{\dom,2}{\widehat{\ul{v}}_h}^2 \\
    &\lesssim h \seminorm{\dom,2}{\widehat{\ul{v}}_h}^2 + h \norm{L^2(\partial \dom)}{\tr (\widehat{\ul{v}}_h)}^2 + h (1 + |\ln h|) \seminorm{\dom,2}{\widehat{\ul{v}}_h}^2.
  \end{align*}

  Arguing as in \cite[Lemma 6.33]{di-pietro.droniou:2020:hybrid} we have $\seminorm{\dom,2}{\widehat{\ul{v}}_h}\lesssim \seminorm{1,h}{\ul{v}_h}$, while noting that $\tr{(\wh{\ul{v}}_h)}$ is the $L^2(\partial\dom)$-projection of $\tr{(\ul{v}_h)}$ onto piecewise constant functions on $\Fhb$, so Jensen's inequality yields $\norm{L^2(\partial \dom)}{\tr (\widehat{\ul{v}}_h)}\le \norm{L^2(\partial \dom)}{\tr (\ul{v}_h)}$. Therefore,
  \begin{equation}
    \sum_{f \in \G_r} |f|_{d-1} |\ol{v}_f|^2 \lesssim h \norm{L^2(\partial \dom)}{\tr (\ul{v}_h)}^2 + h (1 + |\ln h|) \seminorm{1,h}{\ul{v}_h}^2.
    \label{eq:sum.Gr}
  \end{equation}
  The boundary Poincaré--Wirtinger inequality from \cite[Lemma 5.5]{Badia-Droniou-Tushar-2024-DiscreteTraceTheory} combined with \eqref{eq:int.F.v} gives
  $$
  \norm{L^2(\partial \dom)}{\tr (\ul{v}_h)} \lesssim \left(1+\frac{\mathrm{diam}(\partial\Omega)^d}{|\Gamma|_{d-1}}\right)^{1/2}\tnorm{1/2,h}{\tr (\ul{v}_h)} \lesssim \seminorm{1,h}{\ul{v}_h}.
  $$
  The final estimate relies on two crucial facts: (i) $|\Gamma|_{d-1}\gtrsim 1$ by Assumption \ref{assum:reg.Gamma} and our parameter choices, and (ii) the trace inequality from Theorem \ref{thm:trace}. Substituting this bound into \eqref{eq:sum.Gr} yields
  \begin{align*}
    \sum_{f \in \G_r} |f|_{d-1} |\ol{v}_f|^2 \lesssim h(1 + |\ln h|)\seminorm{1,h}{\ul{v}_h}^2,
  \end{align*}
  which completes the proof of Proposition \ref{prop:bddc.L2}.
\end{proof}

\section{BDDC preconditioner}\label{sec:BDDC} 

This section presents the \ac{bddc} preconditioner for hybrid methods. The \ac{bddc} method is a non-overlapping domain decomposition method that enforces continuity of certain functionals across subdomain interfaces while maintaining the parallel nature of the algorithm.

\subsection{Functional setting and the condition number estimate}\label{sec:bddc.setting}
Consider a coarse mesh (subdomain partition) $\TH$ obtained by agglomeration of the fine mesh $\Th$. We define the \emph{coarse faces} as the intersections between pairs of subdomains with non-zero $(d-1)$-dimensional measure, i.e.,   
\[
\FH \doteq \{ F = \partial T \cap \partial T' : T, T' \in \TH, T \neq T', |F|_{d-1} > 0\},
\] 
where $|F|_{d-1}$ denotes the $(d-1)$-dimensional measure of the face $F$. The coarse mesh $\MH$ is thus defined as the pair $(\TH, \FH)$, where $\TH$ consists of general polytopal subdomains and the coarse faces in $\FH$ are, in general, not flat, since both are constructed by agglomeration of the fine mesh. We assume that each coarse face $F \in \FH$ satisfies Assumption~\ref{assum:reg.Gamma} with respect to every subdomain $T$ it belongs to. We denote by $\Th(T)$ and $\Fh(T)$ the sets of fine cells and fine faces, respectively, of the fine mesh $\Mh$ that are contained in the subdomain $T \in \TH$; the corresponding local hybrid mesh is denoted by $\Mh(T)$.

The original \emph{fine} (\emph{global}) hybrid space $\Uhc$ has been defined in the previous section. This is the space used for the discretisation of the problem on the fine mesh $\Mh$. In order to define a preconditioner for the corresponding discrete system, we need to introduce auxiliary spaces with \emph{relaxed} continuity.  
Given a subdomain $T \in \TH$, we denote the corresponding \emph{local} hybrid space by $\Uhc(T)$, which is defined as the Cartesian product of \ac{hho} spaces on the local meshes $\Mh(T)$, $T \in \TH$:
$$
\Uhd \doteq \bigtimes_{T \in \TH} \Uhc(T) \simeq \Uhi \times \Uhdb \quad\mbox{ with }\quad \Uhdb \doteq \bigtimes_{T \in \TH} \Uhcb(T).
$$ 
We note that the sub-assembled space is two-valued on faces that lay on subdomain interfaces, i.e., $f \in \FH$; no continuity is enforced across subdomains. For each $\ul{u}_h\in\Uhd$ and $T\in\TH$, we denote by $\ul{u}_h(T)$ the restriction of $\ul{u}_h$ to $T$, that is, its component in $\Uhc(T)$ in the product above.
The \ac{bddc} space $\Uhbddc \subset \Uhd$ is a subspace of the sub-assembled space $\Uhd$ that enforces continuity of certain functionals across coarse faces. At each face ${F} \in \FH$, we define the set of functionals (usually referred to as coarse degrees of freedom) $\Lambda_F: \Uhdb(F) \rightarrow \mathbb{R}$. Specifically, $\ul{u}_h \in \Uhd$ belongs to $\Uhbddc$ if and only if
\[
\lambda_F(\ul{u}_h(T)|_F) = \lambda_F(\ul{u}_h(T')|_F), \; T, T' \text{ neighbours of $F$}, \quad \forall \lambda_F \in \Lambda_F, \ \forall F \in  \FH.
\] 
This continuity only constraints face values on the interfaces. Thus, the BDDC skeleton space $\wt{U}_h^\partial \subset \Uhdb$ can be defined similarly by imposing this continuity condition on the sub-assembled face space $\Uhdb$, while the interior component is identical. In particular, $\Uhbddc = U_h \times \wt{U}_h^\partial$.
Finally, we define the \emph{space of bubbles} as 
$$
\Uhbub \doteq \Uhi \times \Uhbubb 
$$ 
where $\Uhbubb \subset \Uhdb$ is the subspace of skeleton functions in $\Uhdb$ that vanish on the faces in $\FH$, i.e., subdomain interfaces. Note that $\Uhbub \subset \Uhc$. All these spaces are summarised in Table \ref{tab:spaces} and Figure \ref{fig:bddc_spaces}. We note that, in the following, we will abuse notation and use $\Uhc$ also to define the subspace of $\Uhd$ that is isomorphic to $\Uhc$, leaving the isomorphism implicit.

\begin{table}[h]
  \renewcommand*{\arraystretch}{1.4}
  \begin{tabular}{@{}l@{\hspace{1em}}p{0.8\textwidth}@{}}
    \toprule
    \textbf{Space} & \textbf{Continuity Properties at Coarse Face} $F \in \FH$ \\
    \midrule
    $\Uhc$ & Functions are single-valued on each fine face $f \in \Fh(F)$ (standard continuous hybrid space) \\
    \midrule
    $\Uhd$ & Functions admit two distinct values on each fine face $f \in \Fh(F)$ (discontinuous across coarse faces) \\
    \midrule
    $\Uhbddc$ & Functions admit two values on each fine face $f \in \Fh(F)$, with continuity enforced for functionals $\lambda_F \in \Lambda_F$ \\
    \midrule
    $\Uhbub$ & Functions vanish on each fine face $f \in \Fh(F)$ \\
    \bottomrule
  \end{tabular}
  \caption{Continuity properties of the spaces in the BDDC framework.}
  \label{tab:spaces}
\end{table}

\begin{figure}
  \centering
  \newcommand{\relation}[1][0.20em]{
  \kern-#1
  \begin{minipage}[c]{0.016\textwidth}\centering\smash{\large$\supset$}\end{minipage}%
  \kern-#1
}

\noindent

\begin{minipage}[c]{0.23\textwidth}
  \centering
  \begin{tikzpicture}[x=0.8cm,y=0.8cm,baseline=(bb.center)]

    \path[use as bounding box] (0,0) rectangle (4,4);
    \coordinate (bb) at (2,2);

    \draw[thin,lightgray] (0,0.975) -- (2,0.975);
    \draw[thin,lightgray] (2,0.975) -- (4,0.975);
    \draw[thin,lightgray] (0,3.025) -- (2,3.025);
    \draw[thin,lightgray] (2,3.025) -- (4,3.025);
    \draw[thin,lightgray] (0.975,0) -- (0.975,2);
    \draw[thin,lightgray] (0.975,2) -- (0.975,4);
    \draw[thin,lightgray] (3.025,0) -- (3.025,2);
    \draw[thin,lightgray] (3.025,2) -- (3.025,4);

    \draw[thick,black] (0,0) rectangle (4,4);
    \draw[gray,dash pattern= on 1pt off 2pt] (0,2) -- (4,2);
    \draw[gray,dash pattern= on 1pt off 2pt] (2,0) -- (2,4);

    \draw[thick,black,dash pattern= on 4pt off 2pt] (0,2.1) -- (1.9,2.1);
    \draw[thick,black,dash pattern= on 4pt off 2pt] (0,1.9) -- (1.9,1.9);

    \draw[thick,black,dash pattern= on 4pt off 2pt] (2.1,2.1) -- (4,2.1);
    \draw[thick,black,dash pattern= on 4pt off 2pt] (2.1,1.9) -- (4,1.9);

    \draw[thick,black,dash pattern= on 4pt off 2pt] (2.1,0) -- (2.1,1.9);
    \draw[thick,black,dash pattern= on 4pt off 2pt] (1.9,0) -- (1.9,1.9);

    \draw[thick,black,dash pattern= on 4pt off 2pt] (2.1,2.1) -- (2.1,4);
    \draw[thick,black,dash pattern= on 4pt off 2pt] (1.9,2.1) -- (1.9,4);

    \tikzmath{
      let \axa = {0.5,1.45,2.55,3.5};
      let \axb = {0,0.975,3.025,4};
      let \axc = {1.9,2.1};
    }
    \foreach \x in \axa {
      \foreach \y in \axb {
        \draw[fill] (\x,\y) circle (0.05);
        \draw[fill] (\y,\x) circle (0.05);
      }
      \foreach \y in \axc {
        \draw[fill,red!80!black] (\x,\y) circle (0.05);
        \draw[fill,red!80!black] (\y,\x) circle (0.05);
      }
    }
  \end{tikzpicture}

  \vspace{0.25em}\(\Uhd\)
\end{minipage}
\relation
\begin{minipage}[c]{0.23\textwidth}
  \centering
  \begin{tikzpicture}[x=0.8cm,y=0.8cm,baseline=(bb.center)]
    \path[use as bounding box] (0,0) rectangle (4,4);
    \coordinate (bb) at (2,2);

    \draw[thin,lightgray] (0,0.975) -- (2,0.975);
    \draw[thin,lightgray] (2,0.975) -- (4,0.975);
    \draw[thin,lightgray] (0,3.025) -- (2,3.025);
    \draw[thin,lightgray] (2,3.025) -- (4,3.025);
    \draw[thin,lightgray] (0.975,0) -- (0.975,2);
    \draw[thin,lightgray] (0.975,2) -- (0.975,4);
    \draw[thin,lightgray] (3.025,0) -- (3.025,2);
    \draw[thin,lightgray] (3.025,2) -- (3.025,4);

    \draw[thick,black] (0,0) rectangle (4,4);
    \draw[thick,blue!80!black] (0.1,2) -- (1.9,2);
    \draw[thick,blue!80!black] (2.1,2) -- (3.9,2);
    \draw[thick,blue!80!black] (2,0.1) -- (2,1.9);
    \draw[thick,blue!80!black] (2,2.1) -- (2,3.9);

    \draw[thick,black,dash pattern= on 4pt off 2pt] (0,2.1) -- (1.9,2.1);
    \draw[thick,black,dash pattern= on 4pt off 2pt] (0,1.9) -- (1.9,1.9);

    \draw[thick,black,dash pattern= on 4pt off 2pt] (2.1,2.1) -- (4,2.1);
    \draw[thick,black,dash pattern= on 4pt off 2pt] (2.1,1.9) -- (4,1.9);

    \draw[thick,black,dash pattern= on 4pt off 2pt] (2.1,0) -- (2.1,1.9);
    \draw[thick,black,dash pattern= on 4pt off 2pt] (1.9,0) -- (1.9,1.9);

    \draw[thick,black,dash pattern= on 4pt off 2pt] (2.1,2.1) -- (2.1,4);
    \draw[thick,black,dash pattern= on 4pt off 2pt] (1.9,2.1) -- (1.9,4);

    \tikzmath{
      let \axa = {0.5,1.45,2.55,3.5};
      let \axb = {0,0.975,3.025,4};
      let \axc = {1.9,2.1};
    }
    \foreach \x in \axa {
      \foreach \y in \axb {
        \draw[fill] (\x,\y) circle (0.05);
        \draw[fill] (\y,\x) circle (0.05);
      }
      \foreach \y in \axc {
        \draw[fill,red!80!black] (\x,\y) circle (0.05);
        \draw[fill,red!80!black] (\y,\x) circle (0.05);
      }
    }

    \draw[fill,blue!80!black] (0.925,1.95) rectangle (1.025,2.05);
    \draw[fill,blue!80!black] (2.975,1.95) rectangle (3.075,2.05);
    \draw[fill,blue!80!black] (1.95,0.925) rectangle (2.05,1.025);
    \draw[fill,blue!80!black] (1.95,2.975) rectangle (2.05,3.075);
  \end{tikzpicture}

  \vspace{0.25em}\(\Uhbddc\)
\end{minipage}
\relation
\begin{minipage}[c]{0.23\textwidth}
  \centering
  \begin{tikzpicture}[x=0.8cm,y=0.8cm,baseline=(bb.center)]
    \path[use as bounding box] (0,0) rectangle (4,4);
    \coordinate (bb) at (2,2);

    \draw[thin,lightgray] (0,1) -- (2,1);
    \draw[thin,lightgray] (2,1) -- (4,1);
    \draw[thin,lightgray] (0,3) -- (2,3);
    \draw[thin,lightgray] (2,3) -- (4,3);
    \draw[thin,lightgray] (1,0) -- (1,2);
    \draw[thin,lightgray] (1,2) -- (1,4);
    \draw[thin,lightgray] (3,0) -- (3,2);
    \draw[thin,lightgray] (3,2) -- (3,4);

    \draw[thick,black] (0,2) -- (4,2);
    \draw[thick,black] (2,0) -- (2,4);
    \draw[thick,black] (0,0) rectangle (4,4);

    \tikzmath{
      let \axa = {0.5,1.5,2.5,3.5};
      let \axb = {0,1,3,4};
      let \axc = {2};
    }
    \foreach \x in \axa {
      \foreach \y in \axb {
        \draw[fill] (\x,\y) circle (0.05);
        \draw[fill] (\y,\x) circle (0.05);
      }
      \foreach \y in \axc {
        \draw[fill,red!80!black] (\x,\y) circle (0.05);
        \draw[fill,red!80!black] (\y,\x) circle (0.05);
      }
    }
  \end{tikzpicture}

  \vspace{0.25em}\(\Uhc\)
\end{minipage}
\relation
\begin{minipage}[c]{0.23\textwidth}
  \centering
  \begin{tikzpicture}[x=0.8cm,y=0.8cm,baseline=(bb.center)]
    \path[use as bounding box] (0,0) rectangle (4,4);
    \coordinate (bb) at (2,2);

    \draw[thin,lightgray] (0,1) -- (2,1);
    \draw[thin,lightgray] (2,1) -- (4,1);
    \draw[thin,lightgray] (0,3) -- (2,3);
    \draw[thin,lightgray] (2,3) -- (4,3);
    \draw[thin,lightgray] (1,0) -- (1,2);
    \draw[thin,lightgray] (1,2) -- (1,4);
    \draw[thin,lightgray] (3,0) -- (3,2);
    \draw[thin,lightgray] (3,2) -- (3,4);

    \draw[thick,black] (0,2) -- (4,2);
    \draw[thick,black] (2,0) -- (2,4);
    \draw[thick,black] (0,0) rectangle (4,4);

    \tikzmath{
      let \axa = {0.5,1.5,2.5,3.5};
      let \axb = {0,1,3,4};
      let \axc = {2};
    }
    \foreach \x in \axa {
      \foreach \y in \axb {
        \draw[fill] (\x,\y) circle (0.05);
        \draw[fill] (\y,\x) circle (0.05);
      }
      \foreach \y in \axc {
        \node[red!80!black,font=\small] at (\x,\y) {0};
        \node[red!80!black,font=\small] at (\y,\x) {0};
      }
    }
  \end{tikzpicture}

  \vspace{0.25em}\(\Uhbub\)
\end{minipage}
  \caption{Visual representation of the spaces in the BDDC framework. Spaces are represented for order $k = 0$ on a $4 \times 4$ cartesian mesh divided into four $2 \times 2$ subdomains. Interior and interface face DoFs are drawn as black and red dots respectively. Coarse DoFs (constraints) are represented as blue squares. Red zeroes represent the vanishing DoFs (fixed zero value) in the bubble space. Cell DoFs are not present for clarity. From left to right: sub-assembled space $\Uhd$, BDDC space $\Uhbddc$, global hybrid space $\Uhc$, bubble space $\Uhbub$.}
  \label{fig:bddc_spaces}
\end{figure}

Let $\Omega \subset \mathbb{R}^d$, with $d=2, 3$, be a bounded polyhedral domain. We consider the  Poisson equation with homogeneous Dirichlet conditions as a model problem; other boundary conditions can readily be considered with minor modifications. The weak form of the problem reads as follows: find $u \in H_0^{1}(\Omega)$ such that  
\begin{equation}
  \label{eq:continuous.problem}
  \langle \mathcal{A} u, v \rangle \doteq \int_\Omega \alpha \nabla u \cdot \nabla v = \langle f, v \rangle \quad \hbox{for any } v \in H_0^{1}(\Omega), 
\end{equation}
where $f \in H^{-1}(\Omega)$ is the forcing term and $\alpha \in L^\infty(\Omega)$ is a physical coefficient such that $\alpha\ge \alpha_0$ almost everywhere in $\Omega$, where $\alpha_0>0$ is a constant.
 In the analysis below, we consider the constant coefficient case for simplicity in the analysis. However, the condition number bounds below can be extended to problems with high-contrast interfaces, e.g., multi-material problems, after a suitable redefinition of the weighting operator and standard analysis arguments in domain decomposition methods \cite{Dryja-NumerMathHeidelb-1996v,Badia-ApplMathLett-2019y}. We numerically explore problems with high-contrast interfaces in Section \ref{sec:jumping.coefficients}.

Consider a hybrid discretisation $\A_h$ of $\A$ on $\Uhd$ (see Section \ref{sec:disc.skel.methods} for more details). Note that $\A_h$ is generally singular (the constant value is not fixed for subdomains not touching $\partial \dom$). As is usually the case in hybrid discretisations, we will assume that this operator is a block-diagonal operator, with blocks $\A_h(T)$ corresponding to each subdomain $T \in \TH$. The operator $\A_h$ can then be written as follows: for $\ul{u}_h \in \Uhd$,
\begin{equation*}
  \langle \A_h \ul{u}_h, \ul{v}_h \rangle = \sum_{T \in \TH} \langle \A_h(T) \ul{u}_h(T), \ul{v}_h(T) \rangle_T, \qquad \forall \ul{v}_h \in \Uhd.
\end{equation*}  
We use $\| \cdot \|_{\A_h}$ to denote the operator seminorm induced by $\A_h$ on $\Uhd$; analogously, we denote by $\| \cdot \|_{\A_h,T}$ the operator seminorm induced by $\A_h(T)$ on $\Uhc(T)$, for each $T \in \TH$.
We assume that, in every subdomain $T \in \TH$, the operator seminorm is equivalent to the discrete $H^1$-seminorm. This assumption holds for most hybrid nonconforming discretisations (see references provided in Section \ref{sec:disc.skel.methods}).

\begin{assumption}\label{assum:norm.equivalence.Ah}
The following norm equivalence holds in every subdomain $T \in \TH$
\begin{equation}\label{eq:norm.equivalence.Ah}
\norm{{\A}_h,T}{\ul{v}_h}^2 \simeq \seminorm{1,h,T}{\ul{v}_h}^2, \quad \forall \ul{v}_h \in \Uhd(T), \quad \forall T \in \TH.
\end{equation}
\end{assumption}


Let $\wh{\A}_h$ (respectively, $\wt{\A}_h$, $\A_{h,0}$) denote the operators corresponding to the hybrid discretisation of $\A$ on $\Uhc$ (respectively, $\Uhbddc$, $\Uhbub$). The hybrid operator $\wh{\A}_h$ and the linear functional $\ell_h$ define the problem we want to solve, which can be written as $\wh{\A}_h \ul{u}_h = \ell_h$. The operator $\wt{\A}_h$ is an essential ingredient of the BDDC preconditioner.

Define the (discontinuous) harmonic space $\ul{H}_h \subset \Uhd$ as the $\A_h$-orthogonal complement of $\Uhbub$. 
Thus, the discrete harmonic extension operator $\Eh: \Uhd \rightarrow {\ul{H}}_h$ is defined by solving independent local hybrid problems on each $T\in\TH$: for $w_h\in\Uhbb$, $\Eh(w_h)\in \ul{H}_h\subset \Uhd$ is such that
\begin{equation}\label{def:Eh}
\mathcal{E}_h w_h = w_h \mbox{ on }  \FH \qquad\mbox{ and }\qquad\langle \A_{h} \Eh (w_h), \ul{\xi}_h \rangle=0
\quad \forall \ul{\xi}_h \in \Uhbub.
\end{equation}
Thus, the harmonic spaces are uniquely determined by skeleton DOFs on $\FH$. We define the restriction $\wh{\mathcal{E}}_h$ and $\wt{\mathcal{E}}_h$  of $\Eh$ onto $\Uhc$ and $\Uhbddc$, respectively. 

Consider a \emph{weighting} operator $W_h: \wt{U}_h^\partial \rightarrow \Uhcb(\FH)$. For hybrid methods, the weighting operator can be defined coarse face-wise by taking the average of the two values at the fine faces. Specifically, at a coarse face $F \in \FH$ shared by subdomains $T, \, T' \in \TH$, the operator $W_h$ is defined as follows: given $\ul{u}_h \in \Uhd$, 
\begin{align}\label{def:weightingOp}
  {(W_h(\ul{u}_h)(T))_f} = {(W_h(\ul{u}_h)(T'))_f} = \frac{1}{2}({(\ul{u}_h(T))_f+(\ul{u}_h(T'))_f}) \quad \forall f\in\Fh(F),
\end{align}
while it preserves the values on the fine faces and cell values. Naturally, $W_h: \Uhd \rightarrow \Uhc$. 

\begin{figure}
  \centering
  \begin{tikzpicture}[x=0.7cm,y=0.7cm]
  \draw[thin,lightgray] (0,0.975) -- (2,0.975);
  \draw[thin,lightgray] (2,0.975) -- (4,0.975);
  \draw[thin,lightgray] (0,3.025) -- (2,3.025);
  \draw[thin,lightgray] (2,3.025) -- (4,3.025);
  \draw[thin,lightgray] (0.975,0) -- (0.975,2);
  \draw[thin,lightgray] (0.975,2) -- (0.975,4);
  \draw[thin,lightgray] (3.025,0) -- (3.025,2);
  \draw[thin,lightgray] (3.025,2) -- (3.025,4);

  \draw[thick,black] (0,0) rectangle (4,4);
  \draw[gray,dash pattern= on 1pt off 2pt] (0,2) -- (4,2);
  \draw[gray,dash pattern= on 1pt off 2pt] (2,0) -- (2,4);

  \draw[thick,black,dash pattern= on 4pt off 2pt] (0,2.1) -- (1.9,2.1);
  \draw[thick,black,dash pattern= on 4pt off 2pt] (0,1.9) -- (1.9,1.9);

  \draw[thick,black,dash pattern= on 4pt off 2pt] (2.1,2.1) -- (4,2.1);
  \draw[thick,black,dash pattern= on 4pt off 2pt] (2.1,1.9) -- (4,1.9);

  \draw[thick,black,dash pattern= on 4pt off 2pt] (2.1,0) -- (2.1,1.9);
  \draw[thick,black,dash pattern= on 4pt off 2pt] (1.9,0) -- (1.9,1.9);

  \draw[thick,black,dash pattern= on 4pt off 2pt] (2.1,2.1) -- (2.1,4);
  \draw[thick,black,dash pattern= on 4pt off 2pt] (1.9,2.1) -- (1.9,4);

  \tikzmath{
    let \axa = {0.5,1.45,2.55,3.5};
    let \axb = {0,0.975,3.025,4};
    let \axc = {1.9,2.1};
  }
  \foreach \x in \axa {
    \foreach \y in \axc {
      \draw[fill,red!80!black] (\x,\y) circle (0.05);
      \draw[fill,red!80!black] (\y,\x) circle (0.05);
    }
  }
\end{tikzpicture}
\raisebox{3.2em}{$\ \xlongrightarrow{W_h} \ $}
\begin{tikzpicture}[x=0.7cm,y=0.7cm]
  \draw[thin,lightgray] (0,1) -- (2,1);
  \draw[thin,lightgray] (2,1) -- (4,1);
  \draw[thin,lightgray] (0,3) -- (2,3);
  \draw[thin,lightgray] (2,3) -- (4,3);
  \draw[thin,lightgray] (1,0) -- (1,2);
  \draw[thin,lightgray] (1,2) -- (1,4);
  \draw[thin,lightgray] (3,0) -- (3,2);
  \draw[thin,lightgray] (3,2) -- (3,4);

  \draw[thick,black] (0,2) -- (4,2);
  \draw[thick,black] (2,0) -- (2,4);
  \draw[thick,black] (0,0) rectangle (4,4);

  \tikzmath{
    let \axa = {0.5,1.5,2.5,3.5};
    let \axb = {0,1,3,4};
    let \axc = {2};
  }
  \foreach \x in \axa {
    \foreach \y in \axc {
      \draw[fill,red!80!black] (\x,\y) circle (0.05);
      \draw[fill,red!80!black] (\y,\x) circle (0.05);
    }
  }
\end{tikzpicture}
\raisebox{3.2em}{$\ \xlongrightarrow{\Eh} \ $}
\begin{tikzpicture}[x=0.7cm,y=0.7cm]
  \draw[thin,lightgray] (0,1) -- (2,1);
  \draw[thin,lightgray] (2,1) -- (4,1);
  \draw[thin,lightgray] (0,3) -- (2,3);
  \draw[thin,lightgray] (2,3) -- (4,3);
  \draw[thin,lightgray] (1,0) -- (1,2);
  \draw[thin,lightgray] (1,2) -- (1,4);
  \draw[thin,lightgray] (3,0) -- (3,2);
  \draw[thin,lightgray] (3,2) -- (3,4);

  \draw[thick,black] (0,2) -- (4,2);
  \draw[thick,black] (2,0) -- (2,4);
  \draw[thick,black] (0,0) rectangle (4,4);

  \tikzmath{
    let \axa = {0.5,1.5,2.5,3.5};
    let \axb = {0,1,3,4};
    let \axc = {2};
  }
  \foreach \x in \axa {
    \foreach \y in \axa {
    }
    \foreach \y in \axb {
      \draw[fill,black] (\x,\y) circle (0.05);
      \draw[fill,black] (\y,\x) circle (0.05);
    }
    \foreach \y in \axc {
      \draw[fill,red!80!black] (\x,\y) circle (0.05);
      \draw[fill,red!80!black] (\y,\x) circle (0.05);
    }
  }
\end{tikzpicture}
  \caption{Schematic representation of the operator $\mathcal{Q}_h$. Interior and interface face DoFs are drawn as black and red dots respectively. Cell DoFs are not present for clarity.}
  \label{fig:bddc_operators}
\end{figure}

The operator $\mathcal{Q}_h \doteq \wh{\mathcal{E}}_h W_h: \Uhd \rightarrow \wh{\ul{H}}_h(\doteq \ul{H}_h \cap \Uhc)$ weighs the interface values and then extends them harmonically to the interior. A visual representation of these operators is provided in Figure \ref{fig:bddc_operators}.
Using these components, we define the BDDC preconditioner $\wh{\mathcal{B}}_h: \Uhc' \rightarrow \Uhc$ as:
$$
\wh{\mathcal{B}}_h \doteq  \mathcal{Q}_h \widetilde{\A}_h^{-1} \mathcal{Q}_h^{\top} 
+ \A_{h,0}^{-1}.
$$

\begin{theorem}[Condition number estimate]\label{thm:bddc}
  Assume that every coarse face $F \in \FH(T)$ satisfies the regularity Assumption \ref{assum:reg.Gamma}, and that the corresponding set of coarse degrees of freedom $\Lambda_F$ includes the mean value functional, defined by
    \begin{equation*}
    w_h\in\Uhdb(F)\mapsto \int_F w_h.
  \end{equation*}
  Assume that the discretisation satisfies the norm equivalence in Assumption \ref{assum:norm.equivalence.Ah}.
  Then, the following condition number bound for the BDDC preconditioner holds:
  $$\kappa(\wh{\mathcal{B}}_h \wh{\A}_h) \lesssim \left(1 + \ln \frac{H}{h} \right)^2.$$
  The hidden constant is independent of $h$, $H$, and the number of subdomains $T$ in the partition $\MH$.
\end{theorem}

\begin{proof} See Section \ref{sec:bddc.analysis}. \end{proof}

\begin{remark}[Application to hybrid methods]\label{rem:applicationtoBDDC}
  Theorem \ref{thm:bddc} applies to a wide class of hybrid numerical methods whose discrete spaces are built from polynomials on cells and faces, possibly with varying polynomial degrees and non-uniform distributions across the mesh; see \cite[Section 2.5]{Badia-Droniou-Tushar-2024-DiscreteTraceTheory}. The result encompasses methods that may omit cell unknowns, such as certain \ac{hho} variants \cite[Section 5.1]{di-pietro.droniou:2020:hybrid}, and is robust with respect to the specific structure of the discrete space. For methods involving only cell unknowns, such as polytopal \ac{dg} methods, the framework can be extended by hybridizing the discretisation at least on coarse faces, following ideas from \cite{dryja-Galvis-Sarkis-2007-bddc-dg}. Thus, the condition number estimate remains valid under these generalizations, provided the key assumptions on the operator discretisation and coarse degrees of freedom in Theorem \ref{thm:bddc} are satisfied, namely Assumptions \ref{assum:reg.Gamma} and \ref{assum:norm.equivalence.Ah}.
\end{remark}

\subsection{Face restriction operator}\label{sec:FaceRestrictionOperator}

Given a subdomain $T \in \TH$ and a coarse face $F \in \FH(T)$, the face restriction operator $\theta_F^T: \Uhd(T) \to {\ul{H}}_h(T)$ is defined as follows: given $\ul{v}_h \in \Uhd(T)$, $\theta_F^T(\ul{v}_h) \in {\ul{H}}_h(T)$ is the unique element that is discretely harmonic (see \eqref{def:Eh}) and satisfies, for all $f \in \Fhb(T)$:
\begin{align*}
  \tr_f(\theta_F^T(\ul{v}_h)) =
    \begin{cases}
      v_f & \text{if } f \in \Fh(F), \\
      0   & \text{if } f \in \Fhb(T) \setminus \Fh(F).
    \end{cases}
\end{align*}
We remark that the operator $\theta_F^T$ depends solely on the face values of $\ul{v}_h$ on the coarse face $F$. With a slight abuse of notation, we will also use $\theta_F^T$ to denote the operator acting directly on these face values.

The following result is a direct consequence of the truncation estimate for piecewise broken polynomial spaces proved in Theorem \ref{thm:TruncationEst}.
\begin{corollary}\label{cor:FaceRestrictionOp}
  Let $F \in \FH(T)$ satisfy the regularity Assumption \ref{assum:reg.Gamma}, and $\ul{v}_h \in \Uhd(T)$ be such that
    $\int_F \tr(\ul{v}_h) = 0$.
Then, the following estimate holds:
$$\seminorm{1/2,h,\partial T}{\tr_{\partial T}(\theta_F^T(\ul{v}_h))} \lesssim \left(1 + {\ln\frac{H}{h}}\right) \seminorm{1,h,T}{\ul{v}_h}.$$
\end{corollary}
The following result facilitates stable propogation between neighbouring subdomains.
\begin{corollary}\label{cor:moving.between.subdomains}
  Let $T,T'$ be two neighbouring subdomains and $F$ be the coarse face common to $T$ and $T'$. Let $\ul{v}_h(T')\in\Uhd(T')$ and denote by $w_{T,F}$ the piecewise polynomial function on $\Fhb(T)$ constructed by extending by $0$ to $\partial T$ the face values of $\ul{v}_h(T')$, that is: $w_{T,F} = \tr_{F}(\ul{v}_h(T'))$ on $F$ and $w_{T,F} = 0$ on $\partial T \backslash F$. Then under Assumption \ref{assum:reg.Gamma} with $\Gamma = F$ and $\int_F \tr_F(\ul{v}_h(T')) = 0$, the following estimate holds:
  \begin{equation}\label{eq:estimate.truncation.transfer}
    \seminorm{1/2,h,\partial T}{w_{T,F}} \lesssim \left(1+\ln\frac{H}{h}\right) \seminorm{1,h,T'}{\ul{v}_h(T')}.
  \end{equation}
\end{corollary}

\begin{proof}
  Considering the definition of $w_{T',F}$, and following same arguments as in the proof of Theorem \ref{thm:TruncationEst} given in Section \ref{sec:Pf.FaceRestrictionOp} (see in particular \eqref{eq:bddc.Hhalf}), the discrete trace seminorm can be rewritten as follows:
  \begin{equation}\label{eq:bound.transfer.TT'}
    \seminorm{1/2,h,\partial T}{w_{T,F}}^2 = \seminorm{1/2,h,F}{\tr_F(\ul{v}_h(T'))}^2 + 2 \underbrace{\sum_{f \in \Fh(F)} \sum_{f' \in \Fhb(T) \backslash \Fh(F)} |f|_{d-1} |f'|_{d-1} \frac{|\ol{v}_f|^2}{\dffp^d}}_{\doteq E_0(T)}.
  \end{equation}
    We then apply \eqref{eq:est.E0} to get
    \begin{equation}\label{eq:est.E0.bound.norms}
    E_0(T)\lesssim \sum_{f \in \Fh(F)} |f|_{d-1} \frac{|\ol{v}_f|^2}{\mbox{dist}(x_f,\partial F)}.
    \end{equation}
    The right-hand side in this upper bound no longer depends on $T$, but only on $F$. We can therefore use the arguments developed after \eqref{eq:est.E0} (see \eqref{eq:bddc.E0.2} and \eqref{eq:bddc.log}) but in $T'$ instead of $T$ (and tracking the scaling due to the size $H$ of $T$, see Remark \ref{rem:Truncation.Scaling}) to infer that this right-hand side is bounded above by $(1+\ln \frac{H}{h})^2\seminorm{1,h,T'}{\ul{v}_h(T')}^2$.
    The proof of  \eqref{eq:estimate.truncation.transfer} is concluded by plugging this bound into \eqref{eq:est.E0.bound.norms}, by using the resulting estimate in \eqref{eq:bound.transfer.TT'}, and by invoking the trace inequality \eqref{eq:trace} in $T'$ to write $\seminorm{1/2,h,F}{\tr_F(\ul{v}_h(T'))}\le \seminorm{1/2,h,\partial T'}{\tr(\ul{v}_h(T'))}\lesssim\seminorm{1,h,T'}{\ul{v}_h(T')}$.
\end{proof}

\subsection{Proof of Theorem \ref{thm:bddc}}\label{sec:bddc.analysis}
The condition number $\kappa$ of the BDDC preconditioned system $\wh{\mathcal{B}}_h \wh{\A}_h$ is defined as 
  \begin{align*}
    \kappa(\wh{\mathcal{B}}_h \wh{\A}_h) = \frac{\mu_{\max}(\wh{\mathcal{B}}_h \wh{\A}_h)}{\mu_{\min}(\wh{\mathcal{B}}_h \wh{\A}_h)},
  \end{align*}
  where $\mu_{\max/\min}$ denotes the maximal/minimal eigenvalue.
  By applying the abstract additive Schwarz theory from \cite[Theorem 15]{Mandel2008}, we obtain the following bound:
  \begin{align}\label{eq:cn.1}
    \frac{\mu_{\max}(\wh{\mathcal{B}}_h \wh{\A}_h)}{\mu_{\min}(\wh{\mathcal{B}}_h \wh{\A}_h)} \leq \max \left\{ 1, \sup_{\ul{v}_h \in \Uhbddc} \frac{ \| \mathcal{Q}_h \ul{v}_h \|_{\A_h}^2}{\|\ul{v}_h \|_{\A_h}^2} \right\}.
  \end{align}
 %
  Let $\wt{\ul{U}}_{h,0}$ denote the subspace of functions in $\Uhbddc$ on which  all functionals in $\bigcup_{F\in\Fh}\Lambda_F$ vanish. The space $\Uhbddc$ admits the decomposition $\Uhbddc = \wt{\ul{U}}_{h,0} \oplus \wt{\ul{H}}_h$, where $\wt{\ul{H}}_h \doteq \ul{H}_h \cap \Uhbddc$ and $\wt{\ul{U}}_{h,0}$ is $\A_h$-orthogonal to $\wt{\ul{H}}_h$. Therefore, for any $\ul{v}_h \in \Uhbddc$, we have $\| \ul{v}_h \|_{\A_h} \geq \| \wt{\mathcal{E}}_h \ul{v}_h\|_{\A_h}$, since the discrete harmonic extension $\wt{\mathcal{E}}_h \ul{v}_h$ minimises the energy norm among all functions with the same degrees of freedom on the interfaces. Using this property, together with the facts that $\mathcal{Q}_h \ul{v}_h = \mathcal{Q}_h \wt{\mathcal{E}}_h \ul{v}_h$ and co-domain of $\wt{\mathcal E}_h$ is the space of harmonic functions, we obtain:
 
  \begin{align}\label{eq:cn.2}
    \sup_{\ul{v}_h \in \Uhbddc} \frac{\|\mathcal{Q}_h \ul{v}_h\|^2_{\A_h}}{\|\ul{v}_h\|^2_{\A_h}} &\leq  \sup_{\ul{v}_h \in \Uhbddc} \frac{\|\mathcal{Q}_h \wt{\mathcal{E}}_h\ul{v}_h\|^2_{\A_h}}{\| \wt{\mathcal{E}}_h \ul{v}_h \|^2_{\A_h}} \leq \sup_{\ul{v}_h \in \wt{\ul{H}}_h} \frac{\|\mathcal{Q}_h \ul{v}_h\|^2_{\A_h}}{\| \ul{v}_h \|^2_{\A_h}} \nonumber\\
    &\lesssim 1 + \sup_{\ul{v}_h \in \wt{\ul{H}}_h} \frac{\|\mathcal{Q}_h \ul{v}_h - \ul{v}_h \|^2_{\A_h}}{\| \ul{v}_h \|^2_{\A_h}}.
  \end{align}
  Therefore, combining \eqref{eq:cn.1} and \eqref{eq:cn.2}, we obtain
  \begin{align}\label{eq:cn.3}
    \frac{\mu_{\max}(\wh{\mathcal{B}}_h \wh{A}_h)}{\mu_{\min}(\wh{\mathcal{B}}_h \wh{A}_h)} \lesssim 1 + \sup_{\ul{v}_h \in \wt{\ul{H}}_h} \frac{\|\mathcal{Q}_h \ul{v}_h - \ul{v}_h\|^2_{{\A}_h}}{\|\ul{v}_h\|_{{\A}_h}^2}.
  \end{align}
  
  Let $\ul{v}_h\in\wt{\ul{H}}_h$ and $\ul{w}_h \doteq \mathcal{Q}_h \ul{v}_h - \ul{v}_h \in \wt{\ul{H}}_h$. 
  It readily follows from the linearity of the operator $\mathcal{Q}_h$ that $\ul{w}_h \in \wt{\ul{H}}_h \cap \wt{\ul{U}}_{h,0}$, i.e., all moments $\{ \Lambda_F \}_{F \in \mathcal{F}_H}$ of $\ul{w}_h$ vanish. Besides, on each subdomain $T \in \mathcal{T}_H$, $\ul{w}_h$ can be decomposed into harmonic face contributions:
  \begin{align}\label{eq:cn.3.5}
    \ul{w}_h(T) = \sum_{F \in \mathcal{F}_H(T)} \theta_{F}^{T}(\ul{w}_h).
  \end{align}
  The numerator in \eqref{eq:cn.3} can then be bounded as follows:
  \begin{align}\label{eq:cn.5}
    \| \mathcal{Q}_h \ul{v}_h - \ul{v}_h \|_{\A_h}^2 = \| \ul{w}_h \|^2_{\A_h} = \sum_{T \in \mathcal{T}_H}^{}  \| \ul{w}_h(T) \|^2_{\A_h,T} \lesssim \sum_{T \in \mathcal{T}_H}^{} \sum_{F \in \mathcal{F}_H(T)} \| \theta^T_F(\ul{w}_h)\|_{\A_h,T}^2,
  \end{align}
  where the last bound relies on \eqref{eq:cn.3.5}, the triangle inequality
  and the fact that, by coarse mesh assumption, $\#\FH(T)\lesssim 1$ for all $T\in\TH$.
  Next, we consider each individual term in the sum. For any $F \in \FH(T)$, denoting by $T'$ the neighbouring subdomain sharing $F$ and recalling that $\ul{w}_h=\mathcal Q_h\ul{v}_h-\ul{v}_h=\wh{\mathcal{E}}_h W_h\ul{v}_h-\ul{v}_h$, the definition \eqref{def:weightingOp} of the weighting operator implies that, for any $\alpha_F \in \mathbb{R}$,
  \begin{align*}
    \tr_{\partial T}(\theta^T_F(\ul{w}_h)) &= \frac{1}{2}\tr_{\partial T}\left(\theta^T_F\left( \ul{v}_h(T')|_F - \ul{v}_h(T)|_F\right) \right) \\
    &= \frac{1}{2}\tr_{\partial T}\left(\theta^T_F\left( (\ul{v}_h(T')|_F- \alpha_F)  - (\ul{v}_h(T)|_F - \alpha_F)\right) \right).
  \end{align*}
 We infer the following bound:
    \begin{align}
     \norm{\A_h,T}{\theta^T_F(\ul{w}_h)}^2 
    &\lesssim   \seminorm{1/2,h,\partial T}{\tr_{\partial T}(\theta^T_F(\ul{w}_h))}^2 \nonumber \\
    &
    \lesssim
    \seminorm{1/2,h,\partial T}{\tr_{\partial T}(\theta^T_F ( \ul{v}_h(T')|_F -\alpha_F))}^2+
    \seminorm{1/2,h,\partial T}{\tr_{\partial T}(\theta^T_F (  \ul{v}_h(T)|_F - \alpha_F ))}^2,
  \label{eq:bbb.7}
  \end{align} 
where the first inequality follows from the norm equivalence \eqref{eq:norm.equivalence.Ah} and the discrete lifting property \eqref{eq:lifting} (which yield, since the harmonic extension minimises the energy norm, $\norm{\A_h,T}{\theta^T_F(\ul{w}_h)}\le \norm{\A_h,T}{\mathcal L_{h,T}(\tr_{\partial T}(\theta^T_F(\ul{w}_h)))}\lesssim \seminorm{1/2,h,\partial T}{\tr_{\partial T}(\theta^T_F(\ul{w}_h))}$).
Since $\ul{v}_h \in \wt{\ul{H}}_h$, the coarse degrees of freedom $\Lambda_F$ applied to $\ul{v}_h$ are continuous across coarse faces. Therefore, since the mean value belongs to $\Lambda_F$ by assumption, we can choose $\alpha_F \doteq \int_F \tr_F(\ul{v}_h(T)) = \int_F \tr_F(\ul{v}_h(T'))$. 
  This choice allows us to invoke Corollary \ref{cor:FaceRestrictionOp} for the second term in \eqref{eq:bbb.7} and Corollary \ref{cor:moving.between.subdomains} for the first term to the function $\ul{v}_h(T')-\alpha_F$. Note that $w_{T,F}$ defined in Corollary \ref{cor:moving.between.subdomains} is then equal to $\tr_{\partial T}(\theta^T_F (  \ul{v}_h(T')|_F - \alpha_F ))$. This leads to
  \begin{align*}
    \norm{\A_h,T}{\theta^T_F(\ul{w}_h)}^2 \lesssim \left(1 + \ln\frac{H}{h}\right)^2 ( \seminorm{1,h,T'}{\ul{v}_h(T')}^2 + \seminorm{1,h,T}{\ul{v}_h(T)}^2).
  \end{align*} 
Combining \eqref{eq:cn.5}, the above estimate and the norm equivalence \eqref{eq:norm.equivalence.Ah}, we obtain
  \begin{align}\label{eq:cn.4}
    \norm{\A_h}{\mathcal{Q}_h \ul{v}_h - \ul{v}_h }^2 &\lesssim \sum_{T \in \mathcal{T}_H}^{} \sum_{F \in \mathcal{F}_H(T)} \left(1 + \ln\frac{H}{h}\right)^2 (\norm{\A_h,T'}{\ul{v}_h(T')}^2 + \norm{\A_h,T}{\ul{v}_h(T)}^2) \nonumber\\
    &\lesssim \left(1 + \ln\frac{H}{h} \right)^2 \norm{\A_h}{\ul{v}_h}^2.
  \end{align}
  The proof is concluded by combining \eqref{eq:cn.3} and \eqref{eq:cn.4}.

\section{Application to discontinuous skeletal methods}\label{sec:applications}

In this section, we present numerical experiments to validate the theoretical results from the previous sections. We divide this section into three subsections. In Section \ref{sec:disc.skel.methods}, we present discontinuous skeletal methods, specifically \ac{hdg} and \ac{hho}, along with their mixed-order variants. Section \ref{sec:Truncation.NumExp} presents numerical experiments to validate the truncation estimates in Theorem \ref{thm:TruncationEst}. Finally, in Sections \ref{sec:WeakScalabilityTests} and \ref{sec:jumping.coefficients}, we present the results of weak scalability tests and the case of discontinuous coefficients validating the condition number bounds of the \ac{bddc} preconditioner for discontinuous skeletal methods. All tests are performed on GADI, the high-performance computing facility at the National Computational Infrastructure (NCI), Australia. The code is developed in Julia using the Gridap finite element library \cite{Verdugo-Badia-2022-Gridap,Badia-Verdugo-2020-Gridap}.

\subsection{Discontinuous skeletal methods}\label{sec:disc.skel.methods}
We adopt the general framework developed in \cite{Di_Pietro-Dong-Kanschat-Rupp-2024-Multigrid} to describe discontinuous skeletal methods spanning \ac{hdg} and \ac{hho} methods. To this end, we also need a vector-valued space 
  $$\bs{W}_h \doteq  \bigtimes_{t \in \Th} \left[\mathbb{P}_k(t)\right]^d. $$
We define element-wise local linear operators
$$ \U_t: \Uhcb|_{\partial t} \longrightarrow \Uhi|_t, \;\; \V_t: L^2(t) \longrightarrow \Uhi|_t, \;\; \Q_t: \Uhcb|_{\partial t} \longrightarrow \bs{W}_h|_t.$$
The action of these operators is defined by solving local problems and influences the design of numerical schemes. These local operators are concatenated to form their global counterparts $\U$, $\V$, and $\Q$. It can be shown that the bulk component $u_h$ can be approximated using the skeleton component $u_{\partial \Omega}$ via the transformation
$$u_h \approx \U u_{\partial \Omega} + \V f,$$ 
if $u_{\partial \Omega}$ satisfies the equivalent statically condensed version of the problem $\wh{\A}_h \ul{u}_h = l_h$,
\begin{align}\label{Ph}
	\langle \wh{\A}_h u_{\partial\dom}, \mu \rangle = \int_{\dom} f \U \mu  \qquad \forall \mu \in \Uhcb.
\end{align}
Here,
\begin{align}\label{bf}
	\langle \wh{\A}_h u_{\partial\dom}, \mu \rangle 
	&= \sum_{t \in \Th} \int_t \Q_t u_{\partial t} \cdot \Q_t \mu + s(u_{\partial t}, \mu) \doteq \sum_{t \in \Th} a_t(u_{\partial t}, \mu). 
\end{align}
The bilinear form \eqref{bf} is elliptic and continuous, and the {\it stabilisation} or {\it penalty} term $s(\cdot,\cdot)$ within it is symmetric positive semi-definite. Without loss of generality, we assume homogeneous Dirichlet boundary conditions on $\wh{\U}_h$, and the implementation of \eqref{Ph} results in a symmetric positive-definite matrix system. The choice of spaces, local operators, and stabilisation term completely characterise a discontinuous skeletal method. 

In our numerical experiments, we test the proposed \ac{bddc} preconditioner $\wh{\mathcal{B}}_h$ on the following hybrid methods:

\begin{itemize}
  \item \textbf{HDG method} \cite{Cockburn-Gopalakrishnan-Lazarov-2009-HDG}: Implemented on simplicial and tetrahedral meshes. The local operators $\U_t$ and $\Q_t$ map $u_{\partial t}$ to $u_t \in \Uhi|_t$ and $\bs{q}_t \in \bs{W}_h|_t$ by solving the following local problems:
    \begin{alignat}{2}
      \int_t \bs{q}_t \cdot \bs{p}_t - \int_t u_t \nabla \cdot \bs{p}_t &= - \int_{\partial t} u_f \bs{p}_t \cdot \bs{n} & \qquad\forall \bs{p}_t \in \bs{W}_h|_t, \nonumber \\
      \int_{\partial t} \tr(\bs{q}_t) \cdot \bs{n} - \int_t \bs{q}_t \cdot \nabla v_t &= 0 & \qquad\forall \bs{q}_t \in \Q_h|_t, \nonumber \\ 
      \tr(\bs{q}_t) &= \bs{q}_t + \tau_t (u_t - u_f) \bs{n}. \label{LDGH_fluxtrace}
    \end{alignat}
    Here, $\bs{n}$ denotes the unit outward normal vector to $\partial t$. The stabilisation term in the bilinear form \eqref{bf} is defined as
    \[
      s(u_{\partial t}, \mu) \doteq \tau_t \int_{\partial t} (\U_t u_{\partial t} - u_{\partial t}) (\U_t \mu - \mu).
    \]
    We set the penalty parameter $\tau_t = 1$ in our experiments for this scheme.

  \item \textbf{HDG+ method} \cite{Du-Sayas-2021-HDGplus}: Used for polygonal or polyhedral meshes. This method follows the HDG framework above, but replaces the scalar bulk space with
    \[
      \Uhi \doteq \bigtimes_{t \in \Th} \mathbb{P}_{k+1}(t).
    \]
    The trace of the double-valued vector function in \eqref{LDGH_fluxtrace} is replaced by
    \[
      \tr(\bs{q}_t) = \bs{q}_t + \tau_t (\Pi^k_{\partial t} u_t - u_f) \bs{n},
    \]
    where the penalty parameter is chosen as $\tau_t \simeq h_t^{-1}$. Here, $\Pi^k_{\partial t}: L^2(\partial t) \rightarrow \mathbb{P}_{k}(\mathcal{F}_t)$ denotes the standard $L^2$-orthogonal projection, with $(\Pi^k_{\partial t}(w))|_f = \Pi^k_f(w|_f)$.

  \item \textbf{HHO method} \cite{di-pietro.ern:2015:hybrid}: For this method, the vector-valued bulk space is replaced with
    \[
      \bs{W}_h = \bigtimes_{t \in \Th} \nabla \mathbb{P}_{k+1}(t).
    \]
    The operator $\Q_t$ solves the local potential reconstruction problem as defined in \cite[(3.8)--(3.9)]{Di_Pietro-Dong-Kanschat-Rupp-2024-Multigrid}, and $\U_t$ solves the local problem described in \cite[(3.11)]{Di_Pietro-Dong-Kanschat-Rupp-2024-Multigrid}. We use the original HHO element-wise stabiliser proposed in \cite[(22)]{DiPietro-Ern-Lemaire-2014-HHOstablisation} and \cite[(2.22)]{di-pietro.droniou:2020:hybrid}. 

  \item \textbf{Mixed-order HHO method} \cite{Cicuttin-Ern-Pignet-2021-HHO-Book}: In this variant, the scalar bulk space is replaced with
    \[
      \Uhi \doteq \bigtimes_{t \in \Th} \mathbb{P}_{k+1}(t).
    \]
    The rest of the framework remains the same as for the HHO method above, except for the element-wise stabilisation term (which gives, after static condensation, the cell contribution to $s$ in \eqref{bf}), which is chosen as
    \[
      h_t^{-1} \int_{\partial t} (\Pi^k_{\partial t}(u_t|_{\partial t} - u_{\partial t})) (\Pi^k_{\partial t}(v_t|_{\partial t} - v_{\partial t})).
    \]
\end{itemize}

Regarding Assumption \ref{assum:norm.equivalence.Ah} for these methods, we refer to \cite[Lemma 2.18 and Assumption 5.9]{di-pietro.droniou:2020:hybrid} for HHO methods, and to \cite{Cockburn.Di-Pietro.ea:15} for HDG methods.

\subsection{Verification of truncation estimates}\label{sec:Truncation.NumExp}

In order to verify the truncation estimate in Theorem \ref{thm:TruncationEst}, we begin by defining the operator $R_{\Gamma}$, which performs the truncation to  $\Gamma \subset \partial \Omega$ of boundary functions as defined in \eqref{bddc.wf}. We have proved that, for any $\ul{v}_h \in \Uh$ with zero mean value on $\Gamma$, 
\begin{equation}\label{eq:RGamma.1}
\tnorm{1/2,h}{R_{\Gamma}({\tr(\ul{v}_h)})} \lesssim \left( 1 + \ln \left( \frac{\mbox{diam}(\dom)}{h} \right)\right) \seminorm{1,h}{\ul{v}_h}.
\end{equation}
Since the discrete harmonic extension $\mathcal E_h$ preserves the boundary values and minimises the discrete $H^1(\dom)$-seminorm, \eqref{eq:RGamma.1} is equivalent to
\begin{equation}\label{eq:RGamma.2}
\tnorm{1/2,h}{R_{\Gamma}(w_h)} \lesssim \left( 1 + \ln \left( \frac{\mbox{diam}(\dom)}{h} \right)\right) \seminorm{1,h}{\mathcal E_h(w_h)}\qquad
\forall w_h\in \Uhbb\text{ s.t. }\int_\Gamma w_h=0.
\end{equation}
Furthermore, using Theorems \ref{thm:trace} and \ref{thm:lifting} we can easily obtain the following seminorm equivalence:
\begin{align}\label{eq:equivalence}
  \seminorm{1,h}{\mathcal{E}_h(w_h)} \lesssim \tnorm{1/2,h}{w_h} \lesssim \seminorm{1,h}{\mathcal{E}_h(w_h)} \qquad \forall w_h \in \Uhbb.
\end{align}
Let $\HhalfOp$ denote the Gram matrix associated with discrete trace seminorm, $\boldsymbol{S} \doteq [\int_{\partial\dom} \phi_a]_{1 \leq a \leq \mbox{dim}(\Uhbb)}$ be a row vector, $\phi_a$ denoting the shape functions of the boundary space $\Uhbb$, and $\hat{\boldsymbol{w}}$ and $\boldsymbol{R}_\Gamma$ be the vector representation of $\hat{w}_h \in \Uhbb$ and the restriction operator, respectively.  Then the relation \eqref{eq:equivalence} shows that \eqref{eq:RGamma.2} is equivalent to
\begin{align}
  \left( 1 + \ln \left( \frac{\mbox{diam}(\dom)}{h} \right)\right) & \gtrsim
   \max_{\substack{w_h \in \Uhbb, \\ \int_{\Gamma} w_h = 0}} \frac{\tnorm{1/2,h}{R_{\Gamma}(w_h)}}{\seminorm{1/2,h}{w_h}}\nonumber \\ 
   & = 
   \max_{\substack{\hat{w}_h \in \Uhbb}} \frac{\tnorm{1/2,h}{R_{\Gamma}(\hat{w}_h - \int_\Gamma \hat{w}_h)}}{\big(\seminorm{1/2,h}{\hat{w}_h}^2 + (\int_\Gamma \hat{w}_h)^2\big)^{1/2}}\nonumber \\ 
   & = \left(\max_{\hat{\boldsymbol{w}}} \frac{\hat{\boldsymbol{w}}^\top  (I - \boldsymbol{1} \boldsymbol{S}_{\Gamma})^\top \boldsymbol{R}_{\Gamma}^{\top} \HhalfOp  \boldsymbol{R}_{\Gamma}(I - \boldsymbol{1} \boldsymbol{S}_{\Gamma}) \hat{\boldsymbol{w}}}{ \hat{\boldsymbol{w}}^\top (\HhalfOp + \boldsymbol{S}_{\Gamma}^\top \boldsymbol{S}_{\Gamma}) \hat{\boldsymbol{w}}}\right)^{1/2}.
\label{eq:truncation.eigenvalue.bound}
\end{align}
The maximum is taken over non-constant boundary functions and, in the second line, we have used the fact that if $\hat{w}_h\in\Uhbb$ and $w_h\doteq\hat{w}_h - \frac{1}{|\Gamma|_{d-1}}\int_\Gamma \hat{w}_h$, then $\int_\Gamma w_h=0$ and $\seminorm{1/2,h}{w_h}=\seminorm{1/2,h}{\hat{w}_h}\le \big(\seminorm{1/2,h}{\hat{w}_h}^2 + (\int_\Gamma \hat{w}_h)^2\big)^{1/2}$. The bound \eqref{eq:truncation.eigenvalue.bound} can be checked by verifying the asymptotic trend as $h\to 0$ of the maximum eigenvalue of
\begin{align}\label{opr}
  (\HhalfOp + \boldsymbol{S}^\top_\Gamma \boldsymbol{S}_\Gamma)^{-1} (I - \boldsymbol{1} \boldsymbol{S}_{\Gamma})^\top \boldsymbol{R}_{\Gamma}^\top \HhalfOp  \boldsymbol{R}_{\Gamma} (I - \boldsymbol{1} \boldsymbol{S}_{\Gamma}).
\end{align}
To numerically assess the behaviour of this eigenvalue, we consider $\Omega=(0,1)^2$ and uniformly discretised Cartesian meshes. We write $\partial\dom$ as a disjoint union of its top $(\Gamma_T)$, bottom $(\Gamma_B)$, left $(\Gamma_L)$ and right $(\Gamma_R)$ parts. In Figure \ref{fig:TruncationEigenvalues}, we compare the maximum eigenvalues of \eqref{opr} against the log of the mesh size $h$ associated with the hybrid discretisation with a planar $\Gamma = \Gamma_T$ and a non-planar $\Gamma = \Gamma_T \cup \Gamma_R$, respectively. 
This figure clearly show that this maximum eigenvalue seems to exactly behave as (and not just be bounded above by) a linear function of $|\ln h|$. In this example on Cartesian meshes, the maximum eigenvalue appears to be independent of the polynomial order $k$; we have not tracked the dependence of the constants on $k$ in the truncation estimate and cannot claim this in general.

\begin{figure}[htbp]
  \centering
  \begin{subfigure}[b]{0.48\linewidth}
    \centering
    \includegraphics[width=\linewidth]{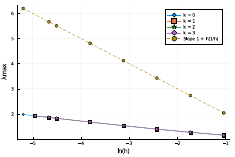}
    \caption{$\Gamma = \Gamma_T$}
    \label{fig:TruncationEigenvalues_planar}
  \end{subfigure}
  \hfill
  \begin{subfigure}[b]{0.48\linewidth}
    \centering
    \includegraphics[width=\linewidth]{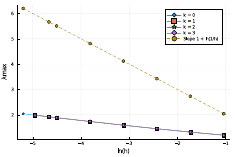}
    \caption{$\Gamma = \Gamma_T \cup \Gamma_R$}
    \label{fig:TruncationEigenvalues_non_planar}
  \end{subfigure}
  \caption{{Maximum} eigenvalues of \eqref{opr} versus $\log(h)$ for polynomial orders $0, 1, 2$ and $3$ with planar and non-planar $\Gamma$.}
  \label{fig:TruncationEigenvalues}
\end{figure}

\subsection{Weak scalability tests}\label{sec:WeakScalabilityTests}

The \ac{bddc} preconditioner exhibits a high degree of parallelism. Combined with the fact that its condition number bounds are independent of the size of the global system, this makes it a truly weakly scalable algorithm in practice. Readers interested in the computational details of its implementation in distributed-memory environments are referred to \cite{Badia-SIAMJSciComput-2016h}. We note that the implementation of the \ac{bddc} preconditioner for hybrid methods is analogous to that for FEM, and the discussion in this reference applies verbatim to the hybrid case.

As a benchmark, we consider the Poisson equation with unit diffusion coefficient, homogeneous Dirichlet boundary conditions, and the following source term:
$$
f(x,y) = -\Delta(\sin(2 \pi x)  \sin(2 \pi y) x (x-1) y (y-1))
$$
on the unit $d$-cube $\left[0,1\right]^d$. We consider two families of meshes: simplicial and polygonal. Each mesh is partitioned into subdomains, with each subdomain assigned to a single processor. We conduct weak scalability tests, where both the global problem size and the number of subdomains (processors) are increased proportionally, keeping the local subdomain problem size fixed.

For a given number of processors $N_p = \prod_{i=1}^d N_p^i$, we construct a Cartesian mesh with $N_c = \prod_{i=1}^d N_c^i$ cells, where $N_c^i = N_p^i \, H/h$ for each $i \in \{1, \dots, d\}$. The simplicial mesh is obtained by subdividing each Cartesian cell into $d!$ simplices, resulting in $N_{c,\mathrm{simplex}} = d! \, N_c$ cells. The polygonal mesh is generated by computing the Voronoi tessellation of the simplexified mesh, associating each Voronoi cell with a Cartesian node, so that $N_{c,\mathrm{voronoi}} = \prod_{i=1}^d (N_c^i + 1)$. Cell ownership in both the simplicial and polygonal meshes is determined by the ownership of the corresponding Cartesian entity (cell or node). Figure~\ref{fig:meshes} illustrates examples of simplicial and polygonal meshes with their subdomain partitions for $d=2$, visualised using ParaView.

We solve the linear systems using the FGMRES iterative solver \cite{Saad-1993-FGMRES}, preconditioned with our \ac{bddc} implementation in GridapSolvers \cite{Manyer-2024-GridapSolvers}. FGMRES is used for implementation convenience, but conjugate gradient could be used as well, since the condensed system matrix is symmetric positive definite. Iterations are performed until the relative residual is reduced below $10^{-8}$, and we record the number of iterations required for convergence.

For $d=2$, we consider $N_p = 6, 12, 24, 48, 96, 192, 384, 768$ processors, and test local problem sizes $H/h = 8$ and $H/h = 16$. Figures~\ref{fig:WeakScalabilityTestsSimplicial} and~\ref{fig:WeakScalabilityTestsPolytopal} report the number of FGMRES iterations for simplicial and polygonal meshes, respectively. For $d=3$, we use $N_p = 8, 16, 32, 64, 128, 256, 512, 768, 1536, 3072$ and $H/h = 4$, with results shown in Figure~\ref{fig:WeakScalabilityTests3D} for tetrahedral meshes.

The results demonstrate that, for fixed $H/h$, the number of iterations remains essentially constant as $N_p$ increases, confirming the weak scalability and the theoretical condition number bounds established in Theorem~\ref{thm:bddc}. We also observe a mild increase in iteration counts as $H/h$ grows, in agreement with the theoretical predictions. Note that the number of global DoFs is the number of processors times the number of DoFs per processor, which is fixed in each plot. Thus, the plots effectively also show how the number of iterations depend on the total number of DoFs in a weak scalability test.

\begin{figure}[!htbp]
  \centering
  \begin{subfigure}[b]{0.3\textwidth}
      \centering
      \includegraphics[width=\textwidth]{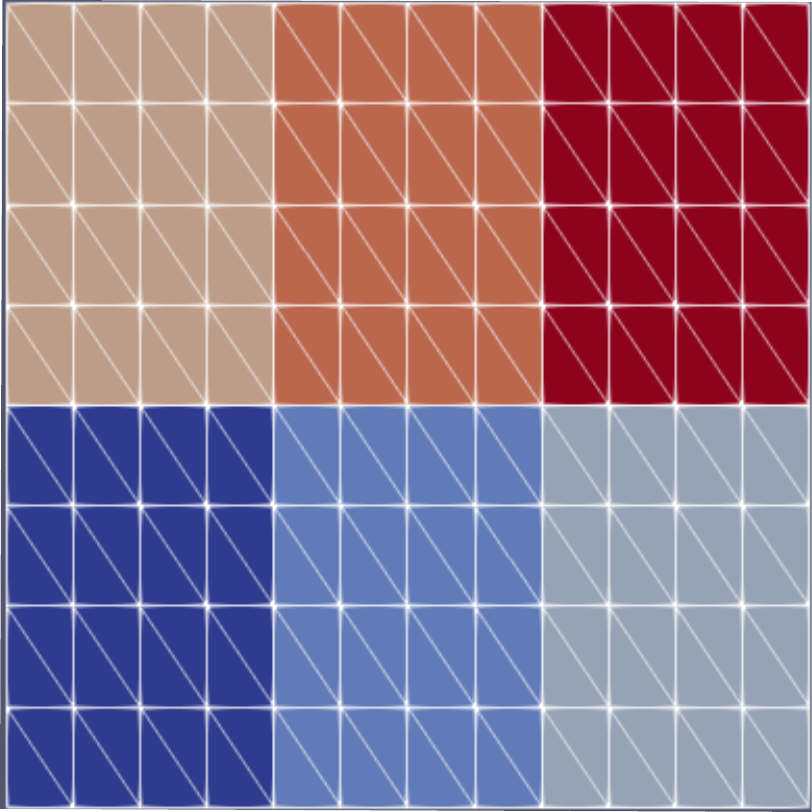}
  \end{subfigure}
  \hfill
  \begin{subfigure}[b]{0.3\textwidth}
      \centering
      \includegraphics[width=\textwidth]{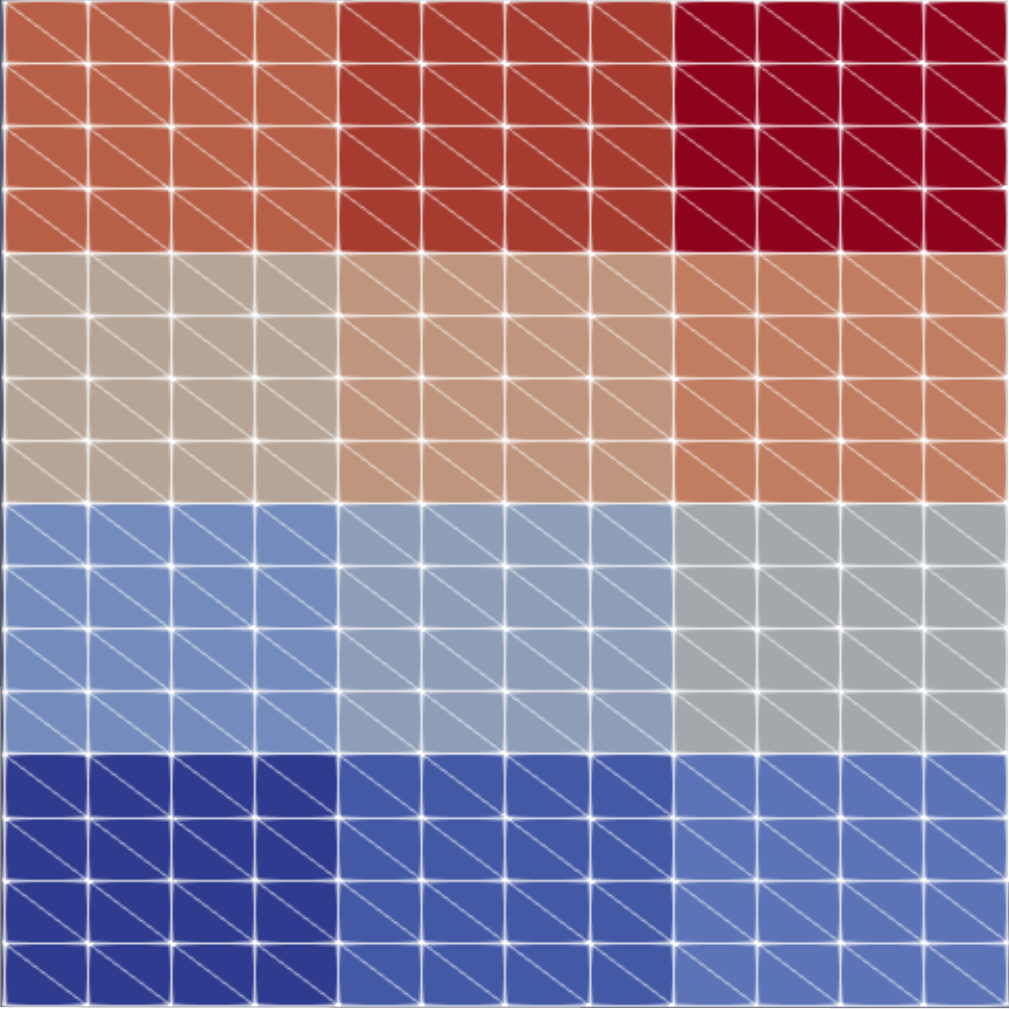}
  \end{subfigure}
  \hfill
  \begin{subfigure}[b]{0.3\textwidth}
      \centering
      \includegraphics[width=\textwidth]{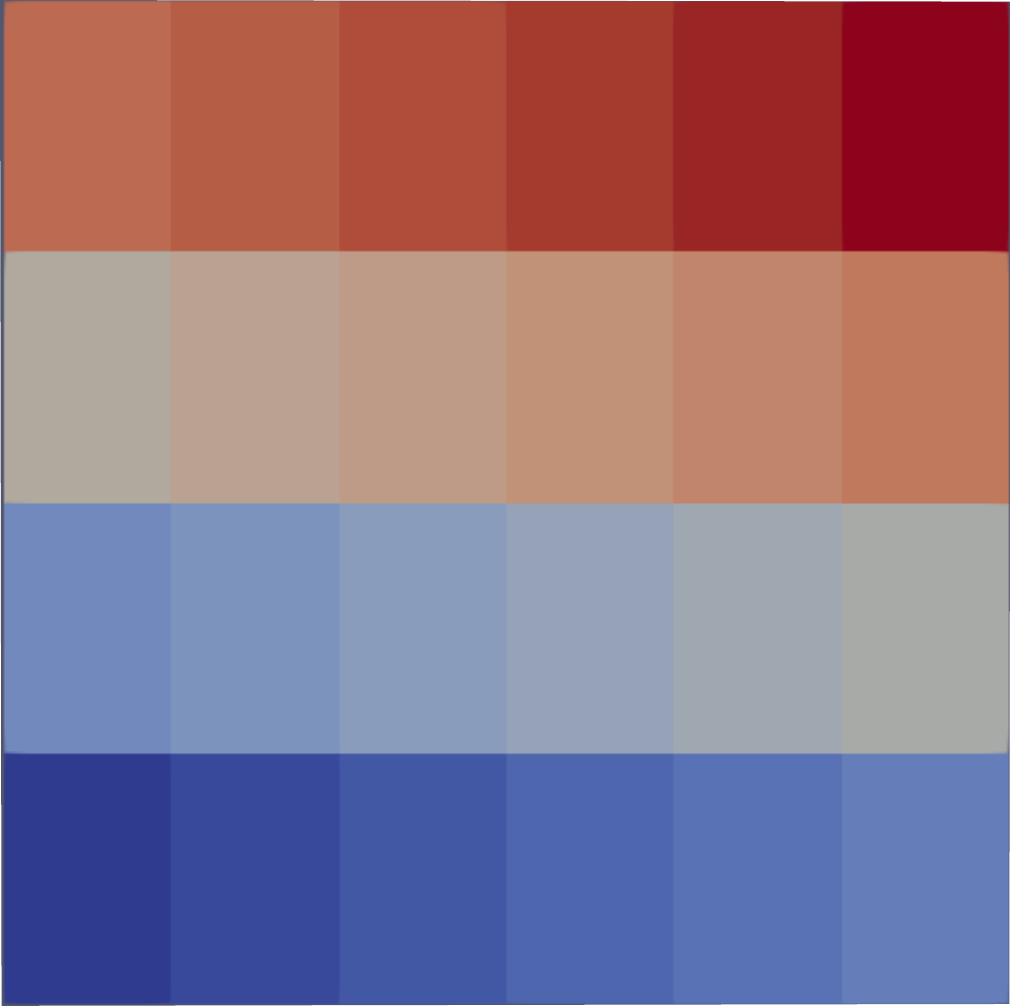}
  \end{subfigure}
  \vskip\baselineskip
  \begin{subfigure}[b]{0.3\textwidth}
      \centering
      \includegraphics[width=\textwidth]{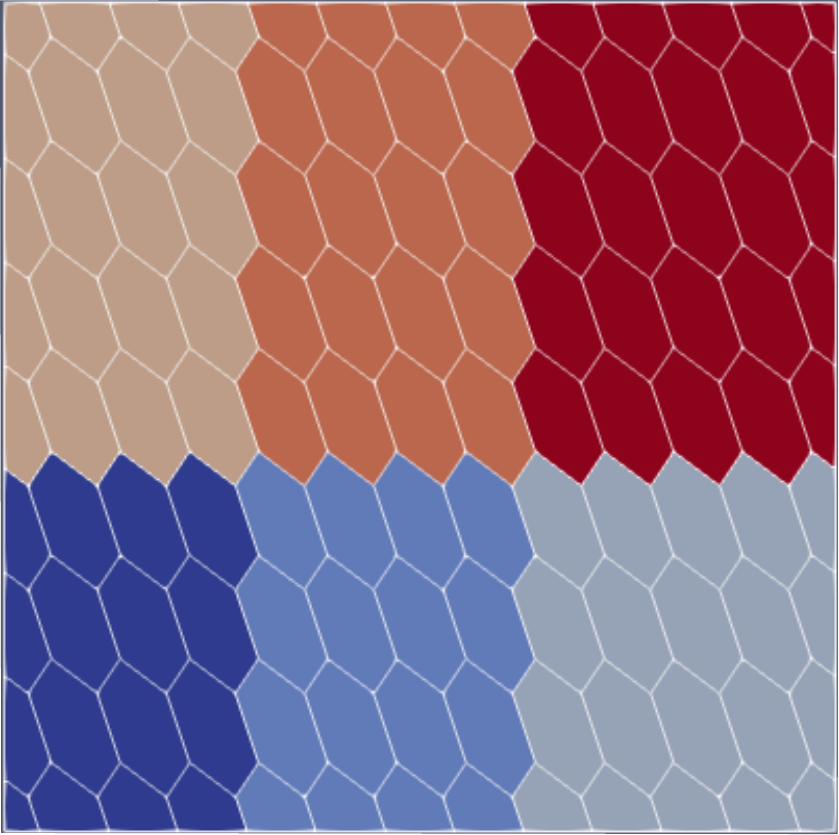}
  \end{subfigure}
  \hfill
  \begin{subfigure}[b]{0.3\textwidth}
      \centering
      \includegraphics[width=\textwidth]{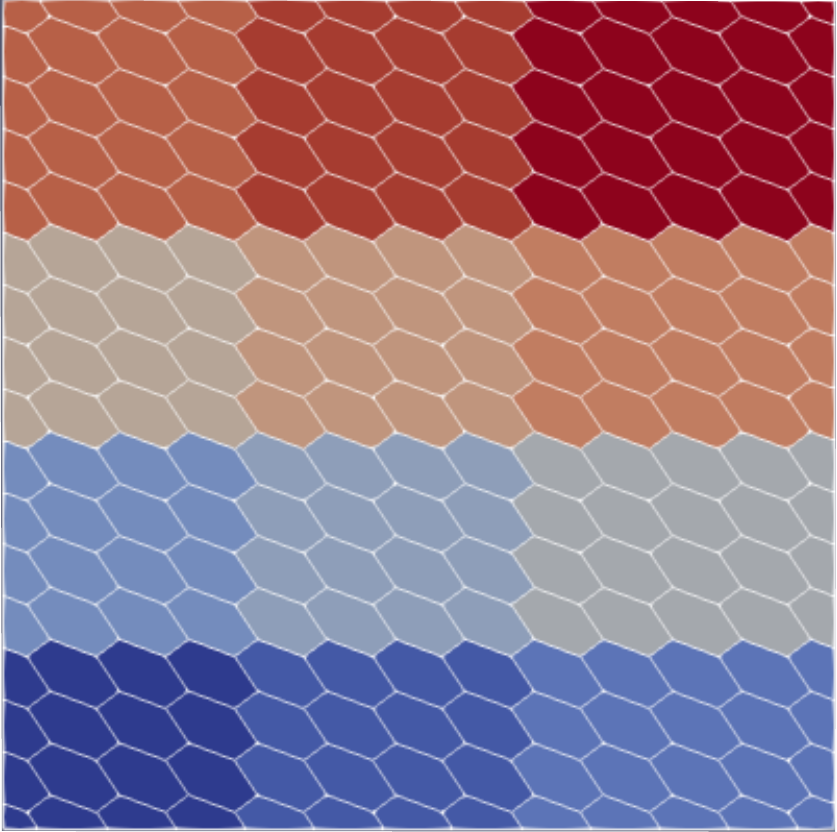}
  \end{subfigure}
  \hfill
  \begin{subfigure}[b]{0.3\textwidth}
      \centering
      \includegraphics[width=\textwidth]{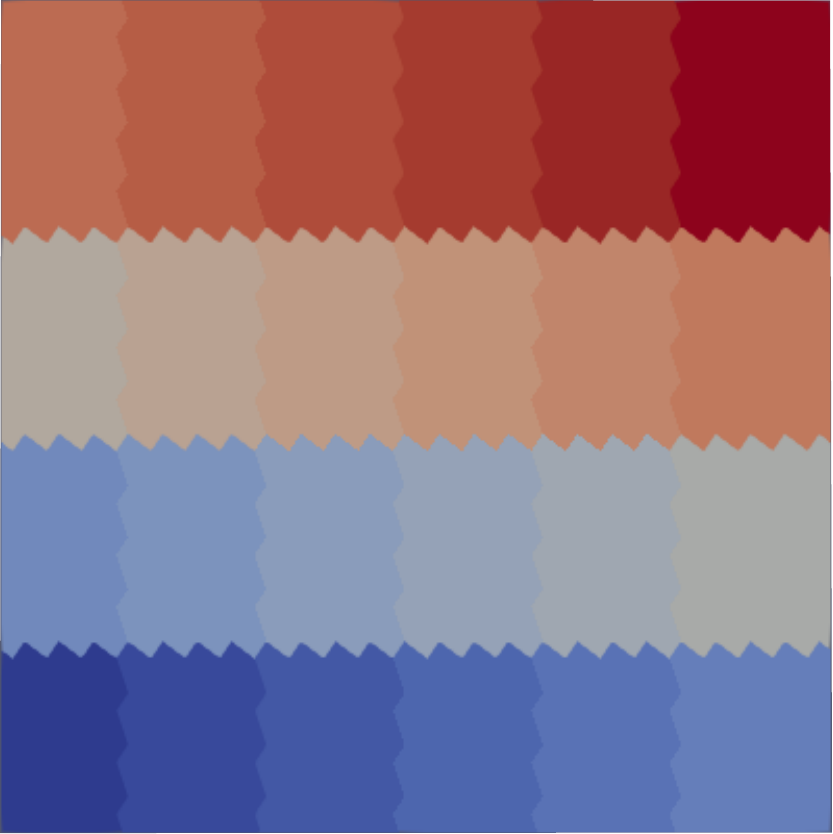}
  \end{subfigure}
  \caption{Simplicial and polygonal meshes with 6, 12, and 24 subdomains.}
  \label{fig:meshes}
\end{figure}

\begin{figure}[!htbp]
  \centering
  \begin{subfigure}[b]{\textwidth}
      \centering
      \includegraphics[width=0.45\textwidth]{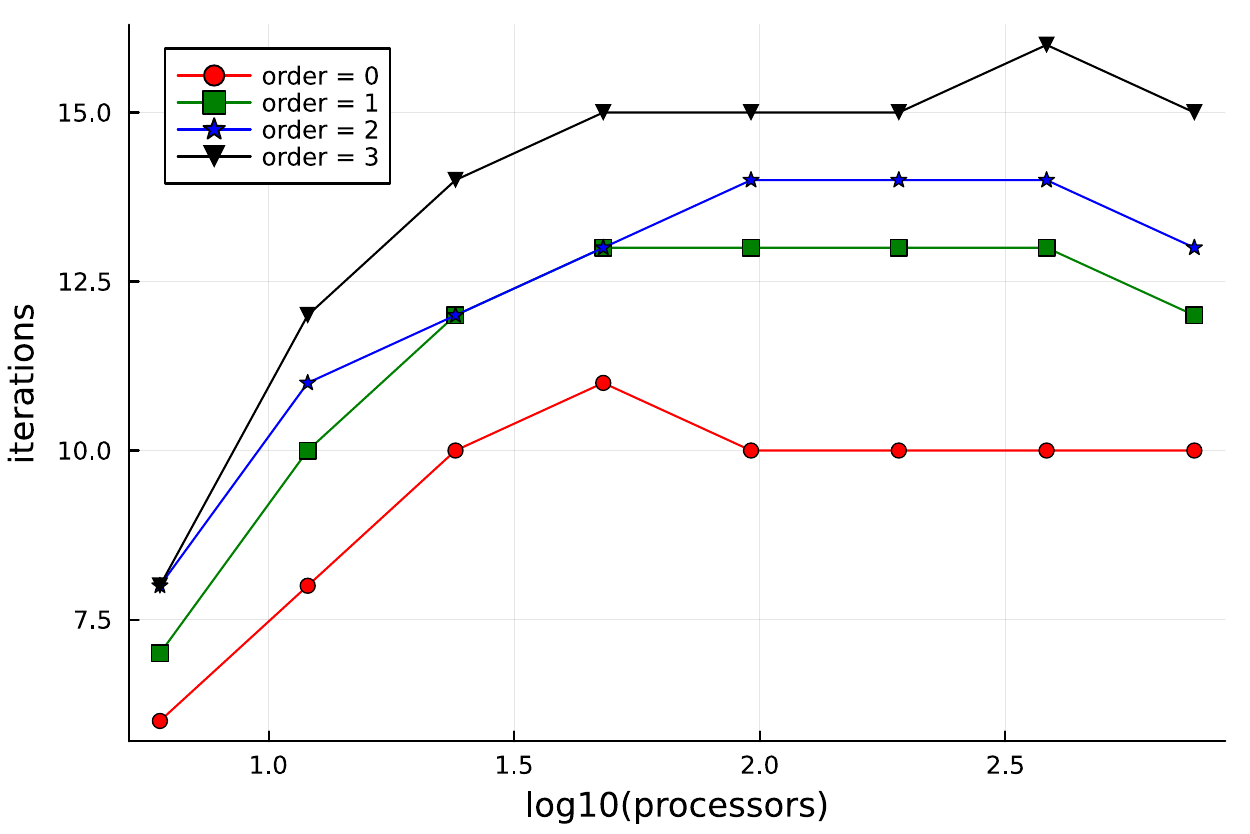}
      \includegraphics[width=0.45\textwidth]{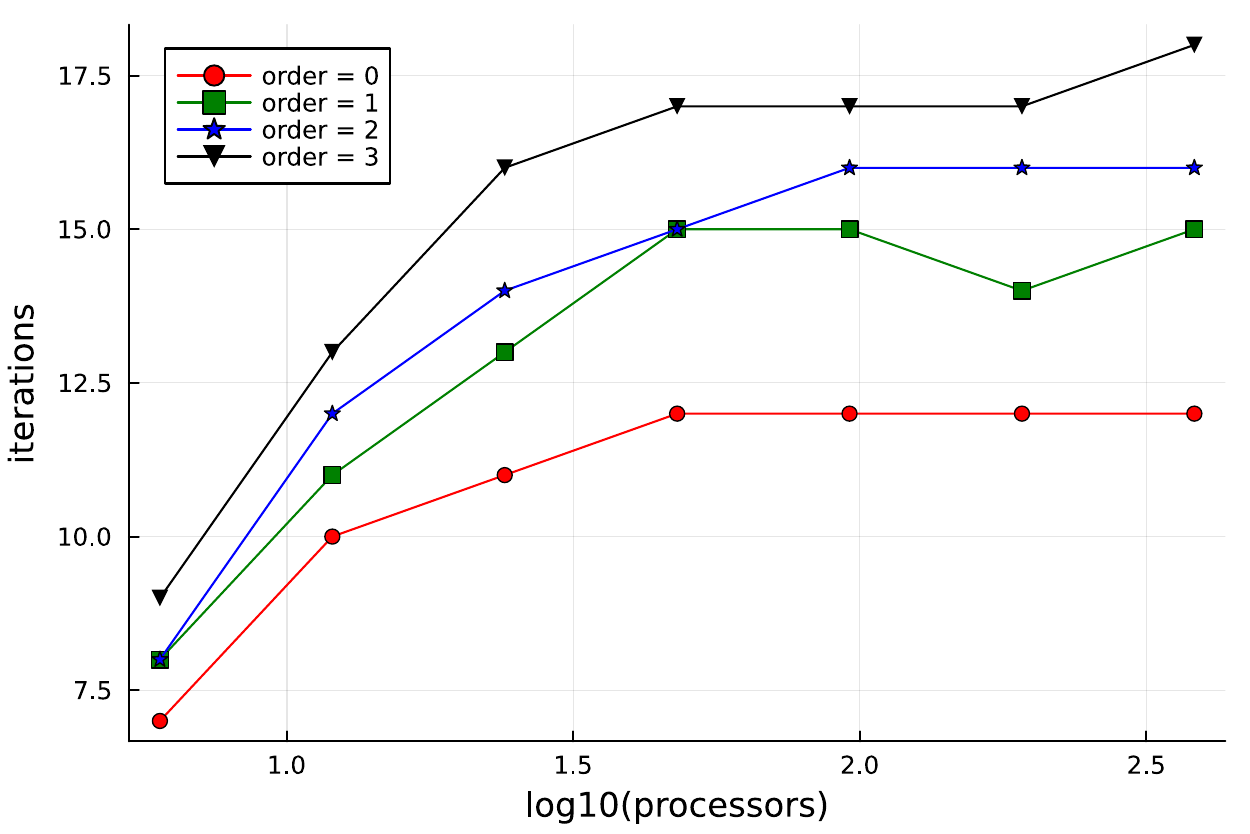}
      \caption{\ac{hdg} on simplicial meshes, for $H/h = 8$ (left) and $H/h = 16$ (right).}
  \end{subfigure}
  \vskip\baselineskip
  \begin{subfigure}[b]{\textwidth}
    \centering
    \includegraphics[width=0.45\textwidth]{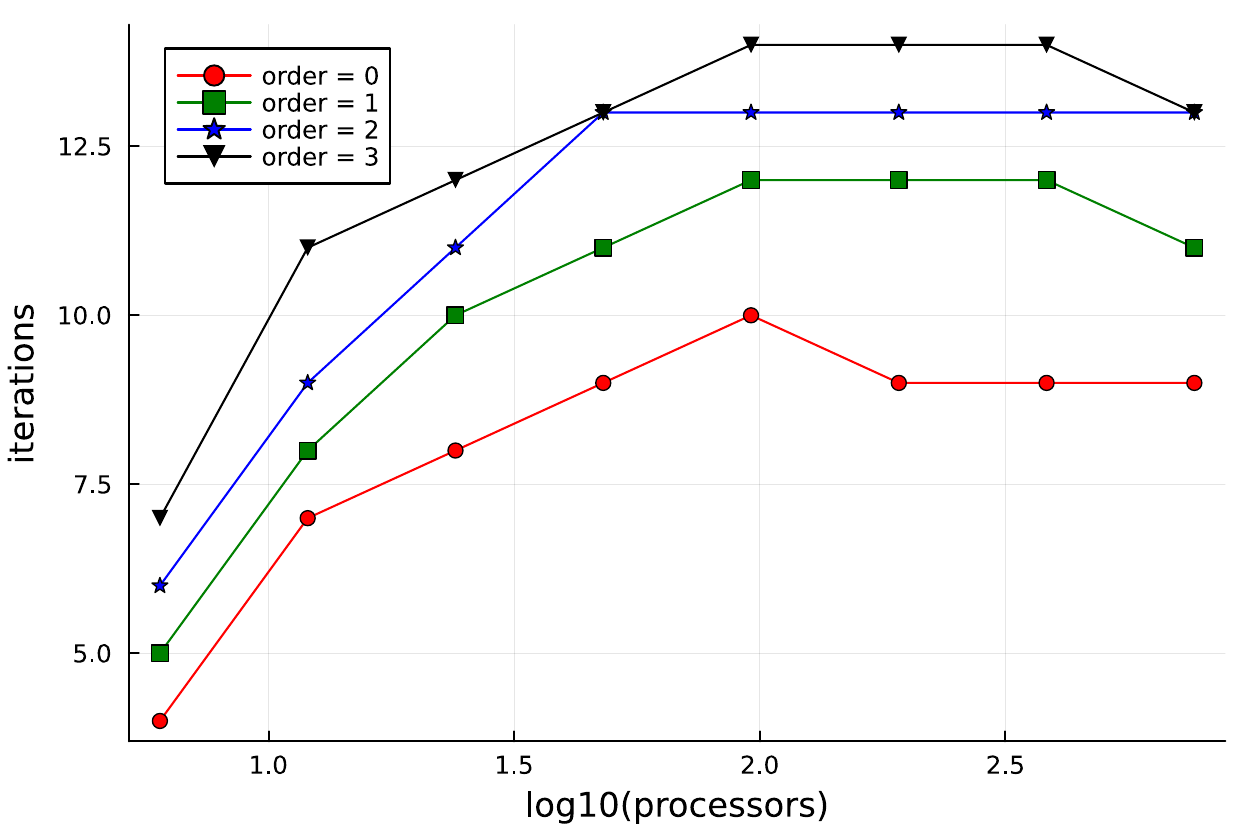}
    \includegraphics[width=0.45\textwidth]{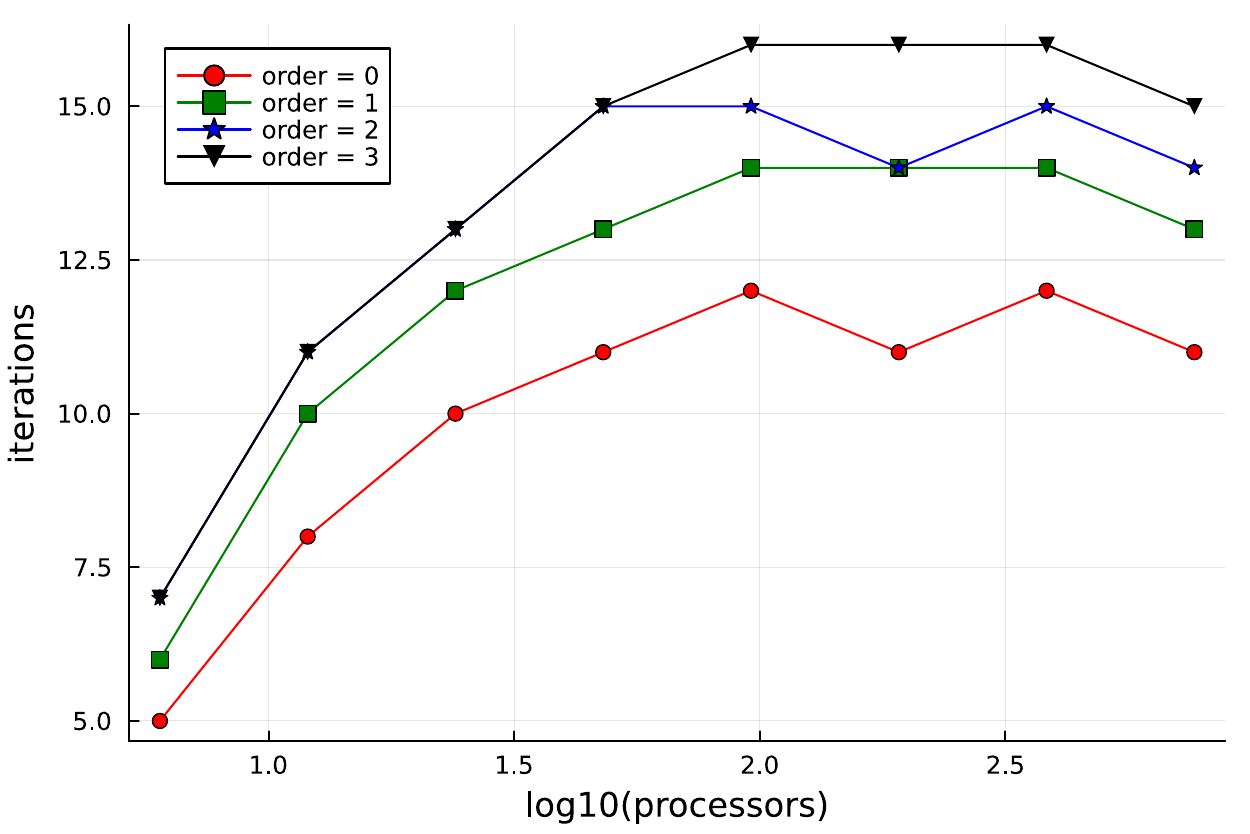}
    \caption{\ac{hho} on simplicial meshes, for $H/h = 8$ (left) and $H/h = 16$ (right).}
  \end{subfigure}
  \vskip\baselineskip
  \begin{subfigure}[b]{\textwidth}
    \centering
    \includegraphics[width=0.45\textwidth]{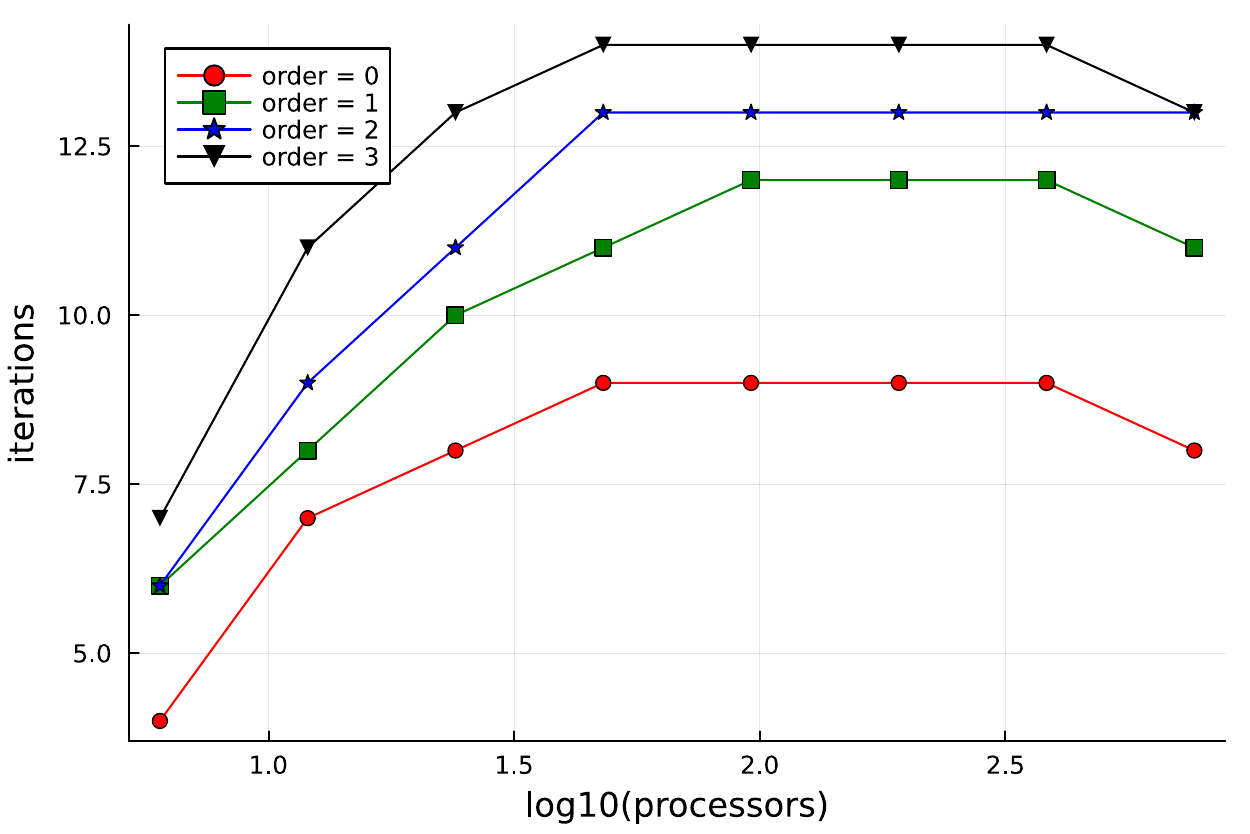}
    \includegraphics[width=0.45\textwidth]{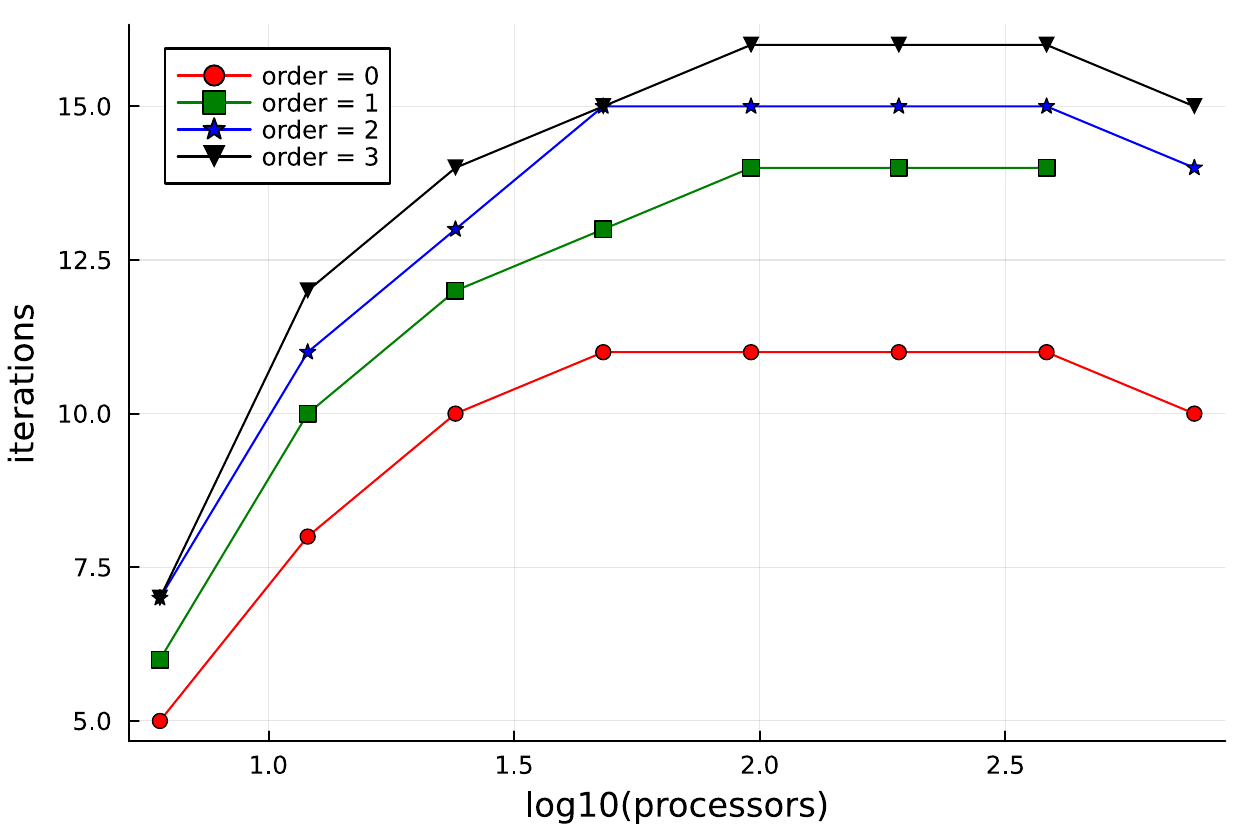}
    \caption{Mixed-order \ac{hho} on simplicial meshes, for $H/h = 8$ (left) and $H/h = 16$ (right).}
  \end{subfigure}
  \caption{Number of processors versus number of FGMRES iterations for different hybrid methods on simplicial meshes and orders $0,1,2,3$.}
  \label{fig:WeakScalabilityTestsSimplicial}
\end{figure}

\begin{figure}[!htbp]
  \centering
  \begin{subfigure}[b]{\textwidth}
    \centering
    \includegraphics[width=0.45\textwidth]{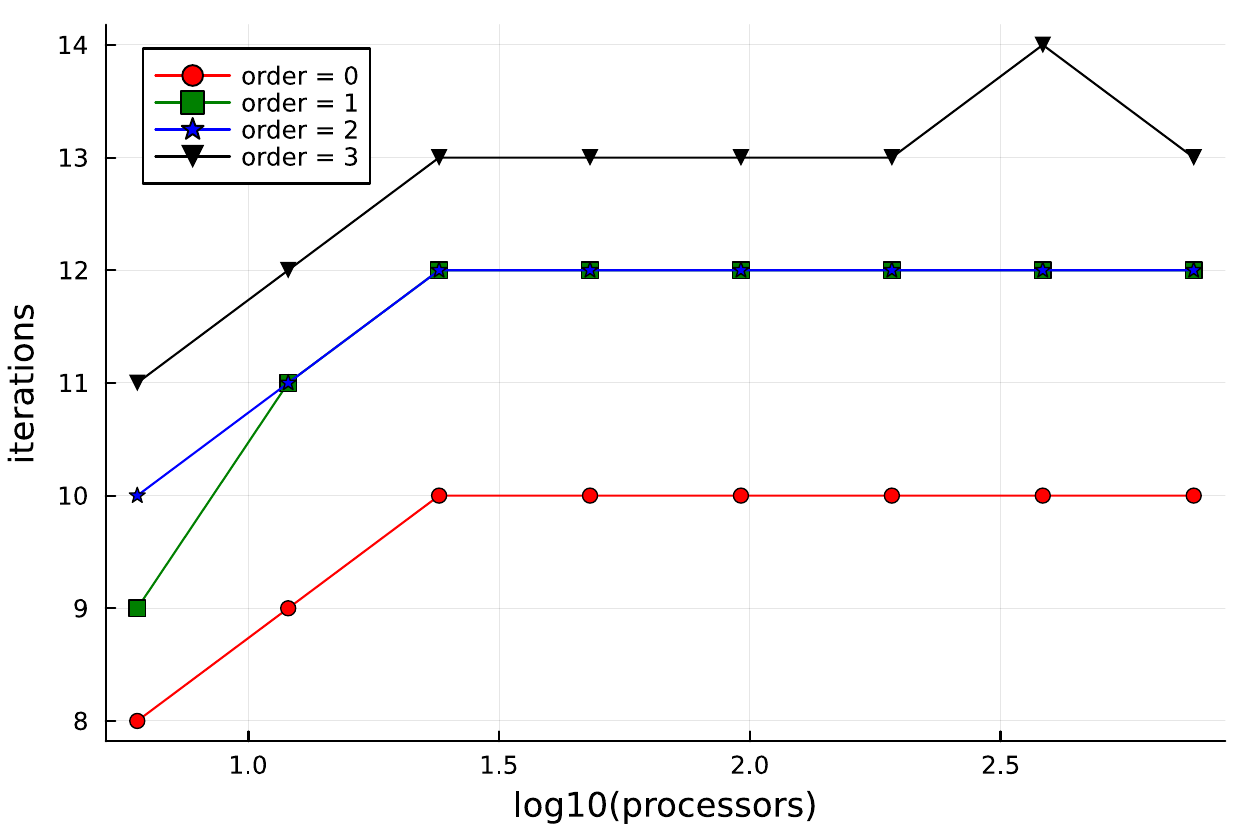}
    \includegraphics[width=0.45\textwidth]{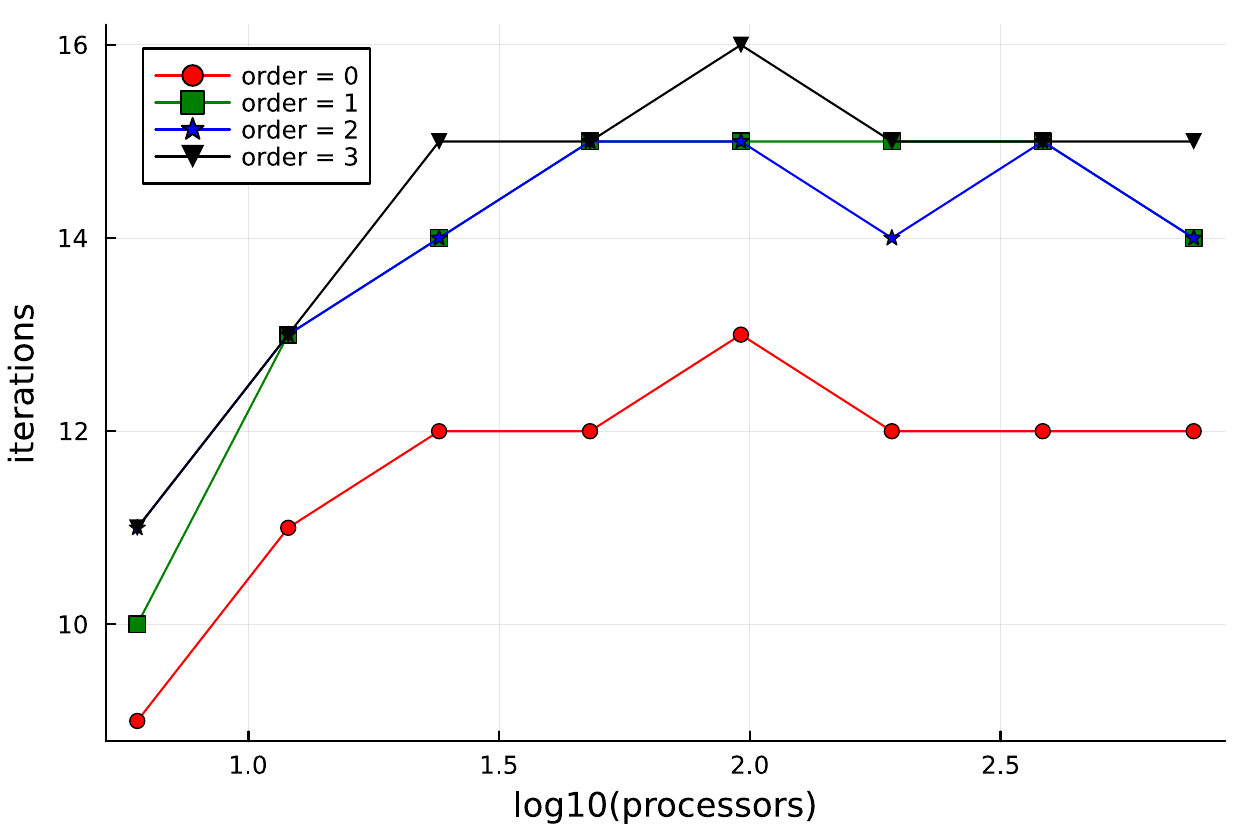}
    \caption{HDG+ on polygonal meshes, for $H/h = 8$ (left) and $H/h = 16$ (right).}
  \end{subfigure}
  \vskip\baselineskip
  \begin{subfigure}[b]{\textwidth}
    \centering
    \includegraphics[width=0.45\textwidth]{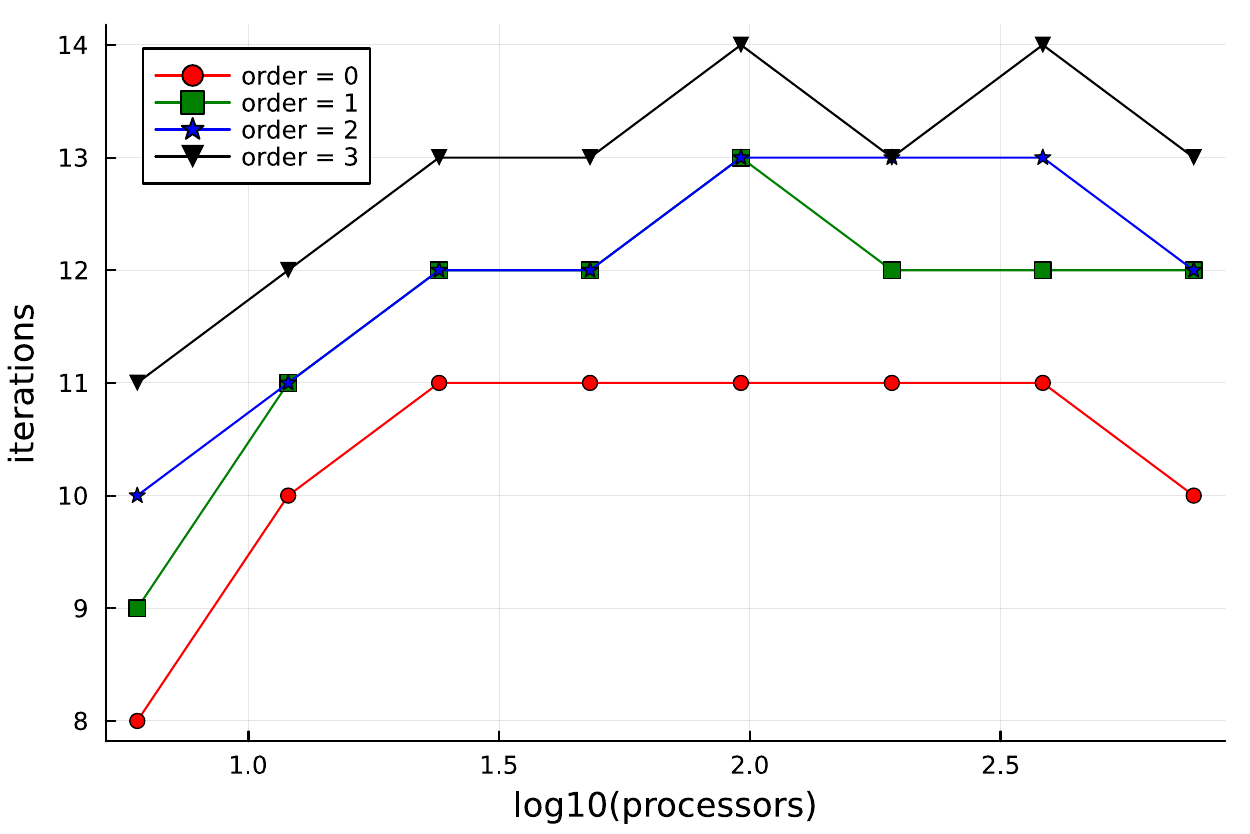}
    \includegraphics[width=0.45\textwidth]{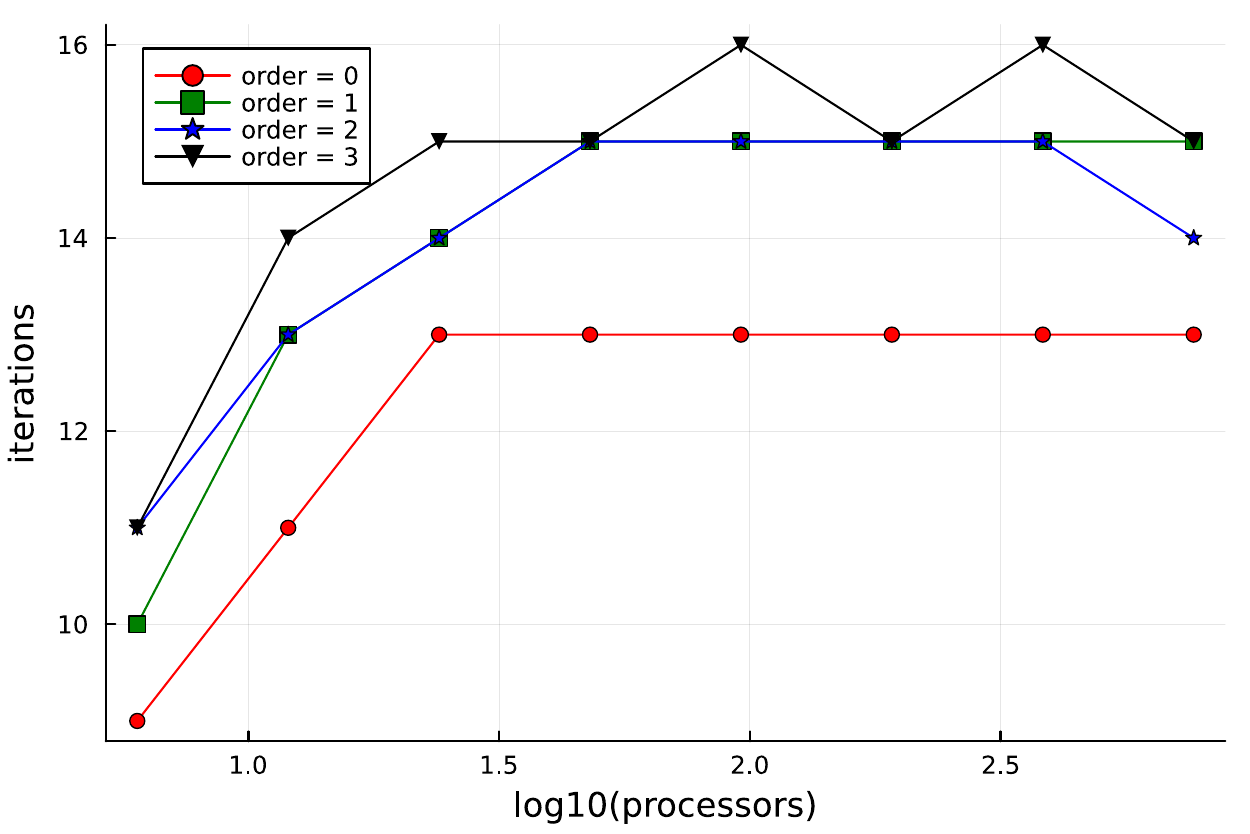}
    \caption{\ac{hho} on polygonal meshes, for $H/h = 8$ (left) and $H/h = 16$ (right).}
  \end{subfigure}
  \vskip\baselineskip
  \begin{subfigure}[b]{\textwidth}
    \centering
    \includegraphics[width=0.45\textwidth]{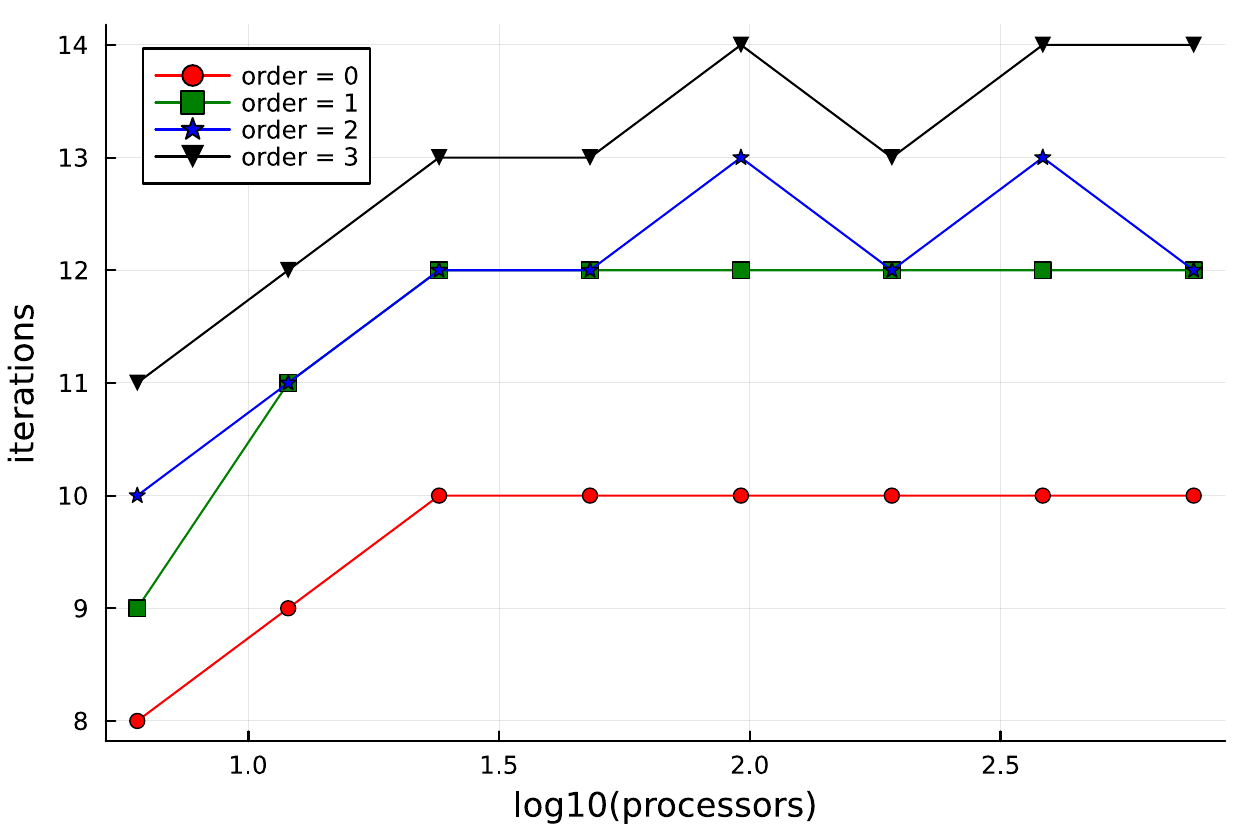}
    \includegraphics[width=0.45\textwidth]{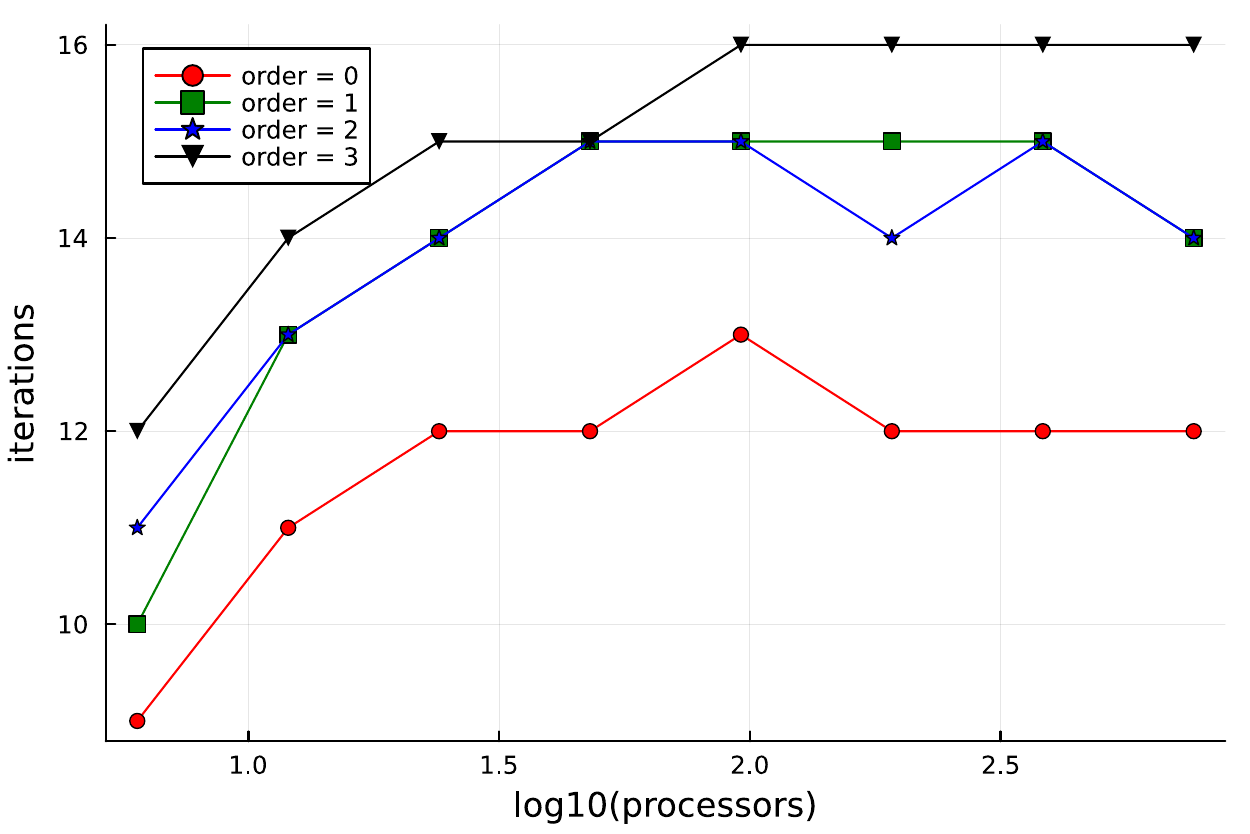}
    \caption{Mixed-order \ac{hho} on polygonal meshes, for $H/h = 8$ (left) and $H/h = 16$ (right).}
  \end{subfigure}

  \caption{Number of processors versus number of FGMRES iterations for different hybrid methods on polygonal meshes and orders $0,1,2,3$.}
  \label{fig:WeakScalabilityTestsPolytopal}
\end{figure}

\begin{figure}[!htbp]
  \centering

  \includegraphics[width=0.45\textwidth]{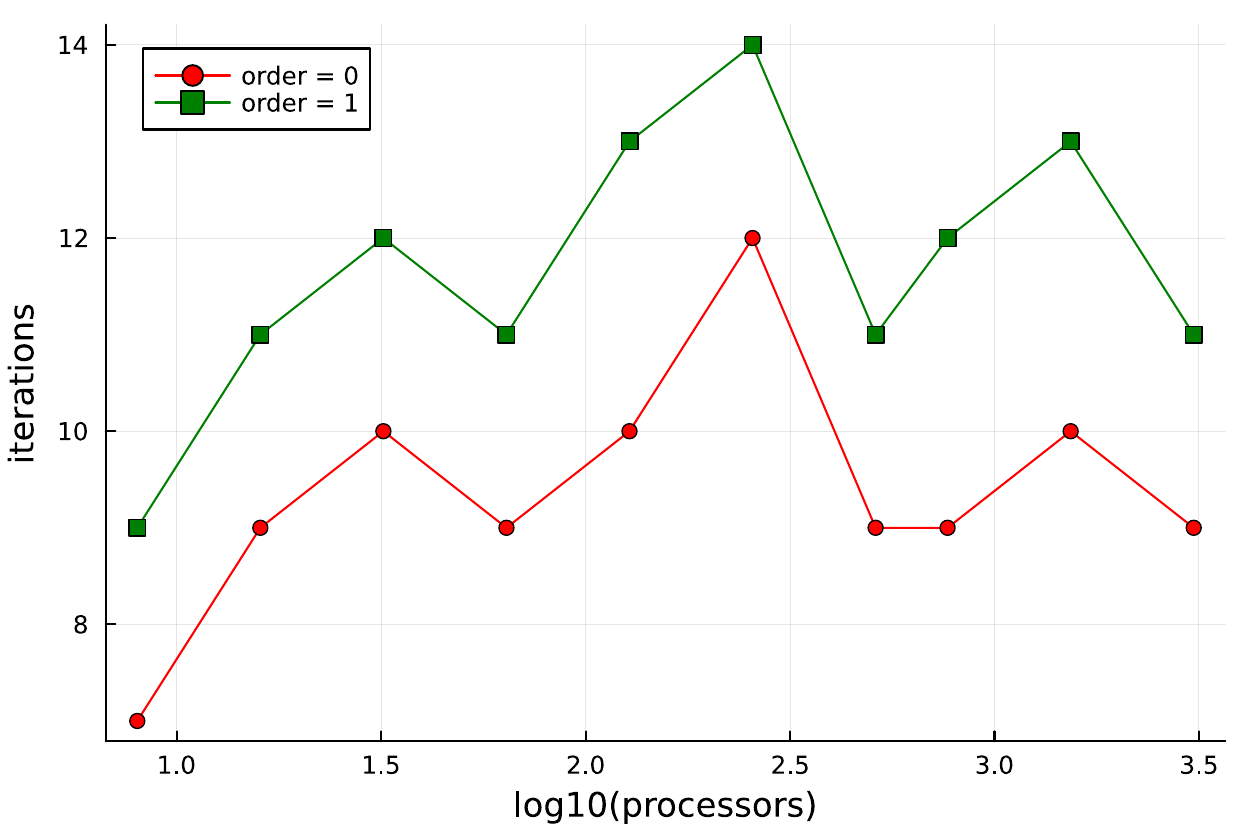}
  \includegraphics[width=0.45\textwidth]{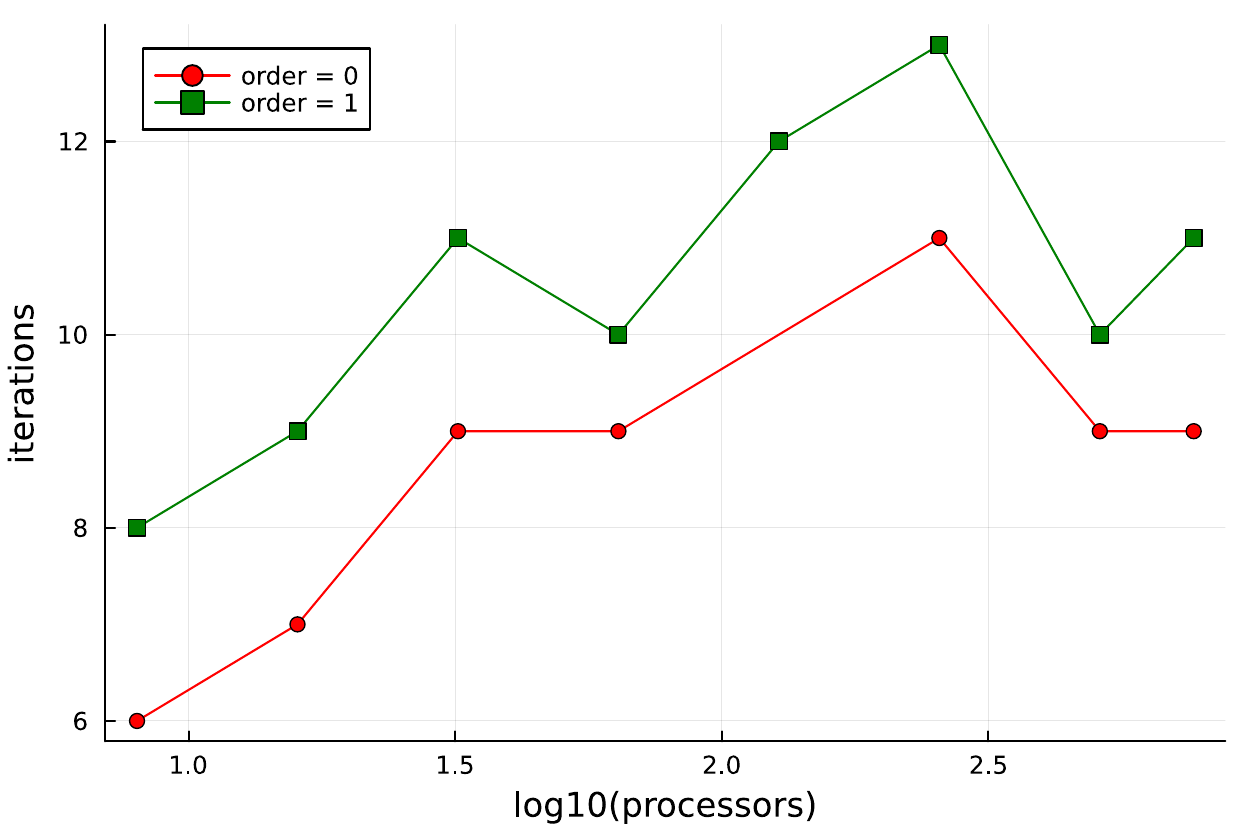}

  \caption{Number of processors versus number of FGMRES iterations for $H/h = 4$, for \ac{hdg} (left) and Mixed-order \ac{hho} (right), on tetrahedral meshes and orders $0,1$.}
  \label{fig:WeakScalabilityTests3D}
\end{figure}

\subsection{Piecewise discontinuous coefficients}\label{sec:jumping.coefficients}

The robustness of the BDDC preconditioner with respect to large jumps in the diffusion coefficient has been demonstrated for standard finite element methods (see, e.g.,~\cite{10.1007/s10915-018-0870-z}). In this case, robustness is achieved by partitioning the mesh into subdomains that align with the discontinuities of the diffusion coefficient and modifying the weighting operator in \eqref{def:weightingOp} to take into account the diffusion coefficients on each side of the subdomain interfaces.

To test the robustness of the proposed BDDC preconditioner for discontinuous skeletal methods in the presence of piecewise discontinuous coefficients, we consider \eqref{eq:continuous.problem} in the case where $\alpha$ is constant in each subdomain but possibly discontinuous accross subdomains, that is, $\alpha|_T = \alpha_T$ is constant for each subdomain $T$. We will limit our experiments to the mixed-order HHO method, modified as in \cite[Section 4.2]{di-pietro.droniou:2020:hybrid} to account for the discontinuous diffusion coefficient. 

We define the modified weighting operator $W_{h,\alpha}:\Uhd \to \Uh$ such that given $\ul{u}_h \in \Uhd$ we have
\begin{align}\label{def:weightingOpJumps}
  {(W_{h,\alpha}(\ul{u}_h)(T))_f} = {(W_h(\ul{u}_h)(T'))_f} = w_{TF} (\ul{u}_h(T))_f + w_{T'F}(\ul{u}_h(T'))_f \quad \forall f\in\Fh(F),
\end{align}
where $T$ and $T'$ are the two subdomains sharing the face $F$, and the weights are defined as
\begin{align*}
  w_{TF} = \frac{\alpha_{T'}}{\alpha_T + \alpha_{T'}}, \quad w_{T'F} = \frac{\alpha_T}{\alpha_T + \alpha_{T'}}.
\end{align*}
We use the same setup as in Section \ref{sec:WeakScalabilityTests}, with $d=2$, $N_p = 16 \times 16 = 256$ processors, subdomain size $H/h=8$, and polynomial order $1$. We consider domain-dependent diffusion where the first $8$ rows of subdomains (roughly $\{ (x,y) \in \Omega : y \leq 0.5\}$) have diffusion coefficient $\alpha_{min} = 1$, and the other $8$ rows have diffusion coefficient $\alpha_{max} =10^k$ for $k=1,\ldots,6$.

\begin{table}
  \begin{tabular}{|c|c|c|}
    $\alpha_{max}$ & Simplicial mesh & Polygonal mesh \\
    \hline
    $10^1$ & 13 & 12 \\
    $10^2$ & 13 & 12 \\
    $10^3$ & 13 & 12 \\
    $10^4$ & 13 & 12 \\
    $10^5$ & 13 & 12 \\
    $10^6$ & 13 & 12 \\
  \end{tabular}
  \caption{Number of FGMRES iterations for mixed-order HHO method with piecewise discontinuous diffusion coefficients on simplicial and polygonal meshes. Fixed parameters: $d=2$, $N_p=256$, $H/h=8$, polynomial order $1$.}
  \label{tab:jumping.coeffs}
\end{table}

The number of FGMRES iterations required for convergence are reported in Table \ref{tab:jumping.coeffs} for both simplicial and polygonal meshes. The results show that the number of iterations remain constant as the jump in the diffusion coefficient increases, demonstrating the robustness of the proposed BDDC preconditioner with respect to piecewise discontinuous coefficients for the mixed-order HHO method.

\section*{Acknowledgments}
This research was supported by the Australian Government through the Australian Research Council (project numbers DP210103092 and DP220103160). Computational resources were provided by the National Computational Infrastructure (NCI) and the Pawsey Supercomputing Research Centre through the NCMAS Merit Allocation Schemes, as well as by Monash eResearch via the Monash NCI scheme. The authors gratefully acknowledge this support.
We also acknowledge the funding of the European Union through the ERC Synergy scheme (NEMESIS, project number 101115663). Views and opinions expressed are however those of the authors only and do not necessarily reflect those of the European Union or the European Research Council Executive Agency. Neither the European Union nor the granting authority can be held responsible for them.

\printbibliography

\end{document}